\documentclass{article}

\usepackage{PRIMEarxiv}

\usepackage[utf8]{inputenc} 
\usepackage[T1]{fontenc}    
\usepackage[colorlinks,linkcolor=blue,citecolor=red,urlcolor=blue]{hyperref}       
\usepackage{url}            
\usepackage{booktabs}       
\usepackage{amsfonts}       
\usepackage{nicefrac}       
\usepackage{microtype}      
\usepackage{lipsum}
\usepackage{fancyhdr}       
\usepackage{graphicx}       
\graphicspath{{media/}}     

\usepackage{amsmath}
\usepackage{amsfonts}
\usepackage{amssymb}
\usepackage{amsthm}
\usepackage[upright]{fourier} 
\usepackage{mathtools}
\usepackage{ifthen}
\usepackage{lettrine}
\usepackage[capitalise]{cleveref} 
\usepackage{tikz}
\usetikzlibrary{calc, math}
\usetikzlibrary{positioning}
\usepackage{tikz-cd}

\newtheorem{theorem}{Theorem}[section]

\newtheorem{proposition}{Proposition}[section]

\theoremstyle{definition}

\newtheorem{assumption}{Assumption}[section]

\theoremstyle{remark}
\newtheorem{remark}{Remark}[section]
\newtheorem{example}{Example}[section]

\numberwithin{equation}{section}

\usepackage{ccicons} 
\usepackage{academicons} 
\usepackage{fontawesome} 
\usepackage{pifont} 
\usepackage{orcidlink}
\usepackage[hyperpageref]{backref}  

\usepackage[numbers]{natbib} 

\newcommand{\mapolicebackref}[1]{
    \hspace*{\fill} \mbox{\textit {\small #1}}
}
\renewcommand*{\backref}[1]{}
\renewcommand*{\backrefalt}[4]{%
\ifcase #1 \mapolicebackref{no quotes}
    \or \mapolicebackref{quoted page #2}
    \else \mapolicebackref{#1 quotes pages #2}
\fi
}

\definecolor{fondtitre}{RGB}{85,85,85}
\definecolor{fonddeboite}{RGB}{232,232,232}
\definecolor{shadecolor}{RGB}{232,232,232}

\title{\Huge\textbf{A constrained Wasserstein gradient flow for mean-field Langevin dynamics}
\thanks{\textit{\underline{Acknowledgements}}: 
\textbf{The author thanks to everyone who has contributed in one way or another to the realization of this work. Last but not least, thankful for the benefits from \href{https://www.lebesgue.fr/en}{\textsc{Henri Lebesgue Center}} (program \textbf{ANR-11-LABX-0020-01}) such as \href{https://www.lebesgue.fr/en/content/bourses_master}{\textsc{Master Scholarships}}, \href{https://www.lebesgue.fr/en/node/4878}{\textsc{Lebesgue Doctoral Meeting}}, \href{https://www.lebesgue.fr/en/content/post-doc}{\textsc{Post doc positions}}... and thankful for the benefits for the ANR project \href{https://math.univ-angers.fr/~chaumont/rawabranch/index.html}{\textsc{Rawabranch}} number ANR-23-CE40-0008.}\newline 
\textit{\underline{Mathematics Subject Classification}}: \textbf{39B62; 82C31; 26D10; 47D07; 60G10; 60H10; 60J60}}
}

\author{
  \Huge{\orcidlink{0000-0002-3281-1954}} \Large{\href{https://mon-portfolio-de-chercheur.webnode.fr/}{Mohamed Alfaki ABOUBACRINE ASSADEK}} \\
   \href{http://www.univ-angers.fr/}{Univ Angers,} \href{https://www.cnrs.fr/fr}{CNRS,} \href{http://recherche.math.univ-angers.fr/}{LAREMA,}\\
   \href{https://sfrmathstic.univ-angers.fr/fr/index.html}{SFR MATHSTIC,}\\
   \href{https://www.angers.fr/}{F-49000 Angers, France }\\
  \texttt{\href{mailto:mohamedalfaki.agaboubacrineassadeck@univ-angers.fr}{mohamedalfaki.agaboubacrineassadeck}@univ-angers.fr} \\
  \AND
  \today\\
}

\begin{document}
\maketitle

\begin{abstract}
In the continuity of \cite{10.1098/rsos.260620,aboubacrineassadek:hal-05236513,ABOUBACRINEASSADEK2026105088}, we study a class of constrained mean-field Langevin dynamics formulated as a Wasserstein gradient flow with an energy constraint. We show that the constrained evolution therefore defines a well‑posed dissipative dynamics in the Wasserstein space $\mathcal{P}_{2}(\mathbb{R}^d)$. We prove that the flow preserves probability mass and the energy constraint, dissipates the free energy monotonically, and converges exponentially to a unique constrained equilibrium under geodesic convexity assumptions. Using synchronous coupling, under Lipschitz assumptions, we demonstrate the quantitative $L^2$-propagation of chaos  with a Fournier-Guillin bound, and under conditions of sufficient dissipativity, we derive an explicit time-uniform bound. We also demonstrate exponential controls in the $\mathcal{W}_2$ Wasserstein metric and the contraction of the invariant measure map. Using log-Sobolev inequalities, we show that the mean-field entropies decay exponentially along the constrained flow. The results are applied to several concrete and original examples.
\end{abstract}

\vfill

\newpage
\tableofcontents

\addcontentsline{toc}{section}{Introduction}
\section*{Introduction}
In this work, among others, we extend the main results of \cite{10.1098/rsos.260620} to general mean-field Langevin dynamics with a general energy constraint. Let us outline what was done in \cite{10.1098/rsos.260620}. The authors studied a class of constrained aggregation–diffusion dynamics (also called granular media dynamics) formulated as a Wasserstein gradient flow with a linear moment constraint. The model describes the evolution of a
probability density on $\mathbb{R}_{+}$ governed by a nonlinear Fokker–Planck equation with confinement potential, two-body interaction and entropic diffusion terms, while preserving a prescribed first moment. The constraint is enforced through a time‑dependent Lagrange multiplier that modifies the drift and keeps the trajectory on an affine feasibility manifold. The authors have shown that the resulting equation arises as the gradient flow of a free‑energy functional in the 2‑Wasserstein space of probability measures subject to a linear constraint. Under standard convexity assumptions on the potentials, the evolution admits global weak solutions, preserves positivity and the moment constraint, satisfies an energy dissipation inequality and converges to a constrained equilibrium. To approximate the dynamics, the authors introduced a
constraint‑preserving variational discretization based on a projected Jordan–Kinderlehrer–Otto (JKO) scheme. The method
maintains mass conservation, enforces the moment constraint
at each time step and inherits the dissipative structure of
the continuous flow. Numerical experiments illustrate convergence towards the constrained equilibrium, verify energy decay and examine the influence of interaction and diffusion parameters. Interpreting the multiplier as a shadow price provides an application to redistribution and adjustment dynamics in heterogeneous‑agent economic systems.\newline

We consider the flow $(\mu_{t})_{t\geq0}$ defined by the constrained nonlinear Fokker–Planck equation:
\begin{equation}\label{gradflowPDE}
    \partial_{t}\mu_{t}=\nabla\cdot\bigg(\mu_{t}\nabla\frac{\delta}{\delta m}\bigg( H+\lambda_{t} E\bigg)(\mu_{t},\cdot)+\sigma^2\nabla\mu_{t}\bigg)
\end{equation}
where $H: \mathcal{P}_{2}(\mathbb{R}^{d})\to(-\infty,+\infty)$ and $E: \mathcal{P}_{2}(\mathbb{R}^{d})\to(-\infty,+\infty)$ are energy functionals admitting flat derivatives, $\frac{\delta}{\delta m}$ is the flat differentiation operator, and the multiplier $\lambda_{t}$ enforces the constraint $E(\mu_{t})=M$, ensuring that the trajectory remains in the feasible set 
\begin{equation}
    \mathcal{C}_{M}:=\bigg\{\mu\in\mathcal{P}_{2}(\mathbb{R}^{d}),\quad E(\mu)=M\bigg\}
\end{equation}
for all time. Define the free‑energy functional $\mathcal{E}:= H+\sigma^2\mathbf{H}[\cdot|dx]$. For $\mu\in\mathcal{C}_{M}$, the tangent space consists of perturbations preserving both total mass and the energy constraint:
\begin{equation}
    T_{\mu}\mathcal{C}_{M}:=\bigg\{\rho\in\mathcal{M}(\mathbb{R}^{d}),\quad\rho(\mathbb{R}^{d})=0\quad\textnormal{and }\int\frac{\delta E}{\delta m}(\mu,x)\rho(dx)=0\bigg\}.
\end{equation}
The constrained Wasserstein gradient flow of $\mathcal{E}$ is the curve defined by $\mu_{t}\in\mathcal{C}_{M}$ satisfying
\begin{equation}\label{gradflow}
    \partial_{t}\mu_{t}=\textbf{argmin}_{\rho\in T_{\mu_{t}}\mathcal{C}_{M}}\bigg\{\int\frac{\delta\mathcal{E}}{\delta m}(\mu_{t},x)\rho(dx)+\frac{1}{2}\|\rho\|^2_{\mathcal{W}_{2}}\bigg\},
\end{equation}
where $\|\rho\|^2_{\mathcal{W}_{2}}$ denotes the metric norm induced by the Wasserstein geometry. The Euler–Lagrange condition associated with equation \cref{gradflow} yields the PDE equation \cref{gradflowPDE}. The additional term $\lambda_{t}\frac{\delta E}{\delta m}(\mu_{t},\cdot)$ arises from projecting the unconstrained gradient of $\mathcal{E}$ onto the tangent space $T_{\mu_{t}}\mathcal{C}_{M}$ in Otto’s Riemannian formulation of optimal transport. In other words, the evolution equation \cref{gradflowPDE} corresponds to the steepest descent of $\mathcal{E}$ restricted to the manifold $\mathcal{C}_{M}$.\newline

If $\lambda_{t}\equiv0$, \cref{gradflowPDE} reduces to the unconstrained nonlinear Fokker-Planck equation associated with the overdamped mean-field Langevin dynamics which is the classical Wasserstein gradient flow of the free energy $\mathcal{E}$ (see, among others, \cite{aboubacrineassadek:tel-05460067,aboubacrineassadek:hal-05236513,ABOUBACRINEASSADEK2026105088,JordanKindrlehererOtto,CarrilloMcCannVillani,Santambrogio,AmbrosioGigliSavare,hu2020meanfieldlangevindynamicsenergy,kazeykina2023ergodicityunderdampedmeanfieldlangevin,Chen_2024,chen2023uniformintimepropagationchaosmean,Monmarch__2024,monmarché2023logarithmicsobolevinequalitiesnonequilibrium,reviewpropachaosI,reviewpropachaosII,GuillinWuZhang,guillin2021convergenceratesvlasovfokkerplanckequation,guillin2019kineticfokkerplanckequationmean,cattiaux2006probabilisticapproachgranularmedia,eberle2017quantitativeharristypetheorems,BolleyGuillinMalrieu,malrieu:hal-01282602}). Regarding constrained flow gradients, a reference article is \cite{EBERLE2017142}. The authors consider a Fokker-Planck equation coupled with an externally imposed, time-dependent constraint on its first moment. This constraint introduces a Lagrange multiplier that renders the equation both nonlocal and nonlinear—specifically, a McKean-Vlasov equation arising from the fact that the Lagrange multiplier is a functional of the solution. In this paper, the authors interpret the equation as a Wasserstein gradient flow of a free energy on a manifold subject to a time-dependent constraint. First, they prove the existence of solutions by passing to the limit in an implicit Euler scheme derived via minimization. Next, they provide quantitative estimates for the rate of convergence to equilibrium as the constraint approaches a constant value. The proof relies on analyzing an appropriate relative entropy with respect to free-energy minimizers determined by the constraint. The convergence rate can be expressed explicitly in terms of constants derived from suitable logarithmic Sobolev inequalities.\newline

\subsection*{\underline{Let us state our main contributions.}}
\begin{itemize}
    \item\textbf{(C1)} We show that the constrained evolution \cref{gradflowPDE} therefore defines a well‑posed dissipative dynamics in the Wasserstein space $\mathcal{P}_{2}(\mathbb{R}^d)$: see \cref{Theorem1}.
    \item\textbf{(C2)} We show that the flow preserves the probability mass and the energy constraint, dissipates the free energy monotonically, and converges exponentially to a unique constrained equilibrium under geodesic convexity assumptions: see \cref{Theorem2} and \cref{Theorem3}.
    \item\textbf{(C3)} Using synchronous coupling, under Lipschitz assumptions, we demonstrate the quantitative $L^2$-propagation of chaos  with a Fournier-Guillin bound, and under conditions of sufficient dissipativity, we derive an explicit time-uniform bound: see \cref{Theorem4}, \cref{Theorem5}, \cref{TheoremBoundLambda}, \cref{TheoremMoments}, \cref{ConseqChaosW2}, \cref{ExempleAppThmBoundLambda}, \cref{FlatDerivativeLambda} and \cref{Thmexplicitformlambda}. Note that \cref{Theorem5} generalizes the analogous result established for the unconstrained case in \cite{aboubacrineassadek:hal-05236513}.
    \item\textbf{(C4)} We also demonstrate exponential controls in the $\mathcal{W}_2$ Wasserstein metric and the contraction of the invariant measure map: see \cref{Theorem6}, \cref{Theorem7} and \cref{Theorem8}. Note that \cref{Theorem7} and \cref{Theorem8} generalize the analogous results established for the unconstrained case in \cite{aboubacrineassadek:hal-05236513}. 
    \item\textbf{(C5)} Using a log-Sobolev inequality, we show that the mean-field entropy decays exponentially along the constrained flow: see \cref{Theorem9} and \cref{PropositionEquilibria}.
    \item\textbf{(C6)} Using a constrained log-Sobolev inequality, we show that the mean-field entropy decays exponentially along the constrained flow: see \cref{Theorem10}, \cref{PropoFreeEnergyLyapunov} and \cref{PropoEntropyProd}.
    \item\textbf{(C7)} Using a constrained log-Sobolev inequality for an explicit invariant measure of the particle system \cref{ConstrainedLangevinSystem}, we show that the microscopic mean-field entropy decays exponentially along the constrained microscopic flow: see \cref{Theorem11}, \cref{ThmExplicitMicroscopicIM} and \cref{propoFreeMicroscopicEnergy}. Note that \cref{Theorem11} is a first step towards extending the results of \cite{ABOUBACRINEASSADEK2026105088} to constrained dynamics, which is what we are doing in a paper currently in preparation.
    \item\textbf{(C8)} The results are applied to several concrete and original examples: see \cref{SectionExamples}.
\end{itemize}
\subsection*{\underline{Let us finish this introduction by the plan of the paper.}} 
The \cref{SectionAssumMainTheorems} presents the main assumptions and the main results. The \cref{Equilibria} characterizes constrained mean-field equilibria. The \cref{SectionParticleApproximation} focuses on approximation by constrained mean-field particle systems. The \cref{SectionCMFEntropy} studies Lyapunov functionals associated with constrained dynamics. The \cref{SectionExamples}  provides examples and applications. The \cref{ProofmainTheorems} contains the proofs of the main theorems. Finally, the paper concludes with the bibliographic references.

\section{Assumptions and main theorems}\label{SectionAssumMainTheorems}
Let $\mathcal{F}: \mathcal{P}_2(\mathbb{R}^d) \to \mathbb{R}$ be a continuous functional. We say that $\mathcal{F}$ admits a flat (or linear) derivative if there exists a measurable function $\frac{\delta \mathcal{F}}{\delta m}: \mathcal{P}_2(\mathbb{R}^d) \times \mathbb{R}^d \to \mathbb{R}$ such that, for all measures $\mu, \nu \in \mathcal{P}_2(\mathbb{R}^d)$, the following integral mean value formula holds:
\begin{equation}
\mathcal{F}(\nu) - \mathcal{F}(\mu) = \int_0^1 \int_{\mathbb{R}^d} \frac{\delta \mathcal{F}}{\delta m}((1-t)\mu + t\nu,x) \, (\nu - \mu)(dx) \, dt.
\end{equation}
The flat derivative is unique up to an additive constant $C_\mu$. Its relationship with the Wasserstein gradient is given by $\nabla_\mathcal{W}\mathcal{F}(\mu,x) =\mathcal{D}_m\mathcal{F}(\mu,x):=\nabla \frac{\delta \mathcal{F}}{\delta m}(\mu,x)$ (Fréchet derivative or intrinsic derivative), provided the spatial derivative exists. To ensure the robustness of the PDE \cref{gradflowPDE} constrained by the energy $E(\mu) = M$, we impose the following minimal geometric and analytical conditions on the flat derivatives of the pair $(\mathcal{E}, E)$:
\begin{itemize}
    \item \textbf{Local regularity and continuity of the flat derivative:}
The mapping $(\mu,x) \mapsto \frac{\delta E}{\delta m}(\mu,x)$ must be continuously differentiable with respect to the spatial variable $x$. Furthermore, its spatial gradient must be locally Lipschitz continuous with respect to the Wasserstein metric:
\begin{equation}
\left\| \nabla\frac{\delta E}{\delta m}(\mu,\cdot) - \nabla\frac{\delta E}{\delta m}(\nu,\cdot) \right\|_{L^2(\mu)} \le L_{E} \mathcal{W}_2(\mu, \nu).
\end{equation}
This minimal regularity ensures that the tangent space to the constraint submanifold $\mathcal{C}_{M}$ varies in a Lipschitz-continuous manner.
\item \textbf{Uniform non-degeneracy of the constraint force field:}
To avoid any singularity or indeterminacy in the orthogonal projection (the calculation of the Lagrange multiplier), the spatial gradient of the flat derivative of the constraint must never vanish on the effective domain of the free energy:
\begin{equation}
\inf_{\mu \in \mathcal{C}_{M}} \int_{\mathbb{R}^d} \left\| \nabla \frac{\delta E}{\delta m}(\mu,x) \right\|^2_2 \mu(dx) \ge c_0 > 0.
\end{equation}
\item \textbf{Growth condition and gradient integrability:}
The spatial gradients of the flat derivatives of the two functionals must satisfy a spatially uniform sub-linear growth condition to stabilize the second-order moments along the flow:
\begin{equation}
\left\| \nabla\frac{\delta \mathcal{E}}{\delta m}(\mu,x) \right\|_2 + \left\| \nabla\frac{\delta E}{\delta m}(\mu,x) \right\|_2 \le C_1 (1 + \|x\|_2), \quad \forall \mu \in \mathcal{C}_{M}.
\end{equation}
\end{itemize}
\textbf{About Lagrange multiplier.} Under these conditions, the Lagrange multiplier $\lambda_\mu$ associated with the energy constraint is rigorously expressed via the flat structure by the functional formula (see \cref{Thmexplicitformlambda}):
\begin{align}
\lambda_\mu&= -\frac{1}{\| \nabla \frac{\delta E}{\delta m}(\mu,\cdot)\|^2_{L^2(\mu)}}\int_{\mathbb{R}^d} \nabla \frac{\delta \mathcal{E}}{\delta m}(\mu,x)\cdot\nabla \frac{\delta E}{\delta m}(\mu,x) \mu(dx)\\
&=\frac{1}{\| \nabla \frac{\delta E}{\delta m}(\mu,\cdot)\|^2_{L^2(\mu)}}\bigg(\sigma^2\int_{\mathbb{R}^d}\Delta\frac{\delta E}{\delta m}(\mu,\cdot)d\mu-\int_{\mathbb{R}^d} \nabla \frac{\delta H}{\delta m}(\mu,\cdot)\cdot\nabla \frac{\delta E}{\delta m}(\mu,\cdot)d\mu\bigg).\nonumber
\end{align}
Note that $\lambda_t=\lambda_{\mu_t}$. The Lipschitz regularity of the flat derivative derived from the first item guarantees that the mapping $\mu \mapsto \lambda_\mu$ is locally Lipschitz on $\mathcal{C}_{M}$ with respect to the Wasserstein $\mathcal{W}_2$ topology.\newline
\textbf{About geodesic convexity.} It is well known that Wasserstein space $(\mathcal{P}_2(\mathbb{R}^d),\mathcal{W}_2)$ is a geodesic space: for any $(\nu_0,\nu_1)\in\mathcal{P}_2(\mathbb{R}^d)\times\mathcal{P}_2(\mathbb{R}^d)$, there exists a curve $\mathcal{G}: [0,1]\longrightarrow\mathcal{P}_2(\mathbb{R}^d)$ such that $\mathcal{G}_0=\nu_0$, $\mathcal{G}_1=\nu_1$ and
\begin{equation}
    \forall (t,s)\in[0,1]\times[0,1],\quad\mathcal{W}_2(\mathcal{G}_t,\mathcal{G}_s)=|t-s|\mathcal{W}_2(\nu_0,\nu_1).
\end{equation}
What makes this fact especially useful is the concrete structure
of these geodesics, also known as displacement interpolations: see Brenier’s theorem in \cite{BrenierMap91}. The geodesics can also be characterized as minimizers (via the continuity equation) in the Benamou-Brenier formula in \cite{BenamouBrenierFormula}. These geodesics were first developed by McCann in \cite{McCannGeodesic97}, to prove the uniqueness of minimizers of certain natural functionals by showing that they are strictly convex along displacement interpolations. Non-geodesically convex submanifolds pose significantly greater difficulties, as the induced geometry and geodesics are genuinely distinct: for further details, see, among others, \cite{chaintron2025geodesicconvexitystrengthenedfunctional}. The geodesic nature of our constraint submanifold $\mathcal{C}_M$ depends on the choice of the energy $E$: for further details, see \cref{SectionExamples}.

\subsection*{\underline{Assumptions}}

\begin{assumption}[\texttt{Regularity}]\label{Assumption1}
    We assume that $\mathcal{E}$ is narrowly lower semi‑continuous and coercive (proper) on the constraint set $\mathcal{C}_{M}$.
\end{assumption}

\begin{assumption}[\texttt{Convexity}]\label{Assumption2}
    We assume that the free energy $\mathcal{E}$ is $\lambda$-geodesically convex along $\mathcal{W}_{2}$‑geodesics in $\mathcal{P}_2(\mathbb{R}^d)$ for some $\lambda\in\mathbb{R}$. That is, for any Wasserstein geodesic $(\nu_{t})_{t\in[0,1]}$ connecting $\nu_0$ and $\nu_1$ in $\mathcal{P}_2(\mathbb{R}^d)$,
    \begin{equation}
        \mathcal{E}(\nu_{t})\leq (1-t)\mathcal{E}(\nu_0)+t\mathcal{E}(\nu_1)-\frac{\lambda}{2}t(1-t)\mathcal{W}^{2}_2(\nu_0,\nu_1).
    \end{equation}
\end{assumption}

\begin{assumption}[\texttt{Lipschitz regularity and first flat derivatives}]\label{Assumption3}
We assume that there are $M^H,M^E,M^E_1\geq0$ such that for all $(x,y)\in\mathbb{R}^d\times\mathbb{R}^d$ and $(\mu,\nu)\in\mathcal{C}_M\times\mathcal{C}_M$,
\begin{align}
    \|\mathcal{D}_mH(\mu,x)-\mathcal{D}_mH(\nu,y)\|_2&\leq M^H(\mathcal{W}_2(\mu,\nu)+\|x-y\|_2);\\
    \|\mathcal{D}_mE(\mu,x)-\mathcal{D}_mE(\nu,y)\|_2&\leq M^E(\mathcal{W}_2(\mu,\nu)+\|x-y\|_2);\\
    \|\mathcal{D}_mE(\mu,x)\|_2&\leq M^E_1.
\end{align}
\end{assumption}

\begin{assumption}[\texttt{Lipschitz regularity and, first and second flat derivatives}]\label{Assumption4}
We assume that there are $\lambda_H,M^E,M^H_2,M^E_1\geq0$  such that for all $(x,y)\in\mathbb{R}^d\times\mathbb{R}^d$ and $(\mu,\nu)\in\mathcal{C}_M\times\mathcal{C}_M$,
\begin{align}
    \langle\mathcal{D}_mH(\mu,x)-\mathcal{D}_mH(\mu,y),x-y\rangle&\geq\lambda_H\|x-y\|^2_2;\\
    \|\mathcal{D}_mE(\mu,x)-\mathcal{D}_mE(\nu,y)\|_2&\leq M^E(\mathcal{W}_2(\mu,\nu)+\|x-y\|_2);\\
    \|\mathcal{D}_mE(\mu,x)\|_2&\leq M^E_1;\\
    \|\mathcal{D}^2_mH(\mu,x,y)\|_{\mathbf{op}}&\leq M^H_2.
\end{align}
\end{assumption}

\begin{assumption}[\texttt{Smallest eigenvalue and the invariant measure application}]\label{Assumption6}
We assume that there are $M^H_2,M^E_1,M^E_2\geq0$ such that for all $(x,y)\in\mathbb{R}^d\times\mathbb{R}^d$ and $\mu\in\mathcal{C}_M$,
\begin{align}
    \|\mathcal{D}^2_mH(\mu,x,y)\|_{\mathbf{op}}&\leq M^H_2;\\
    \|\mathcal{D}_mE(\mu,x)\|_{2}&\leq M^E_1;\\
    \|\mathcal{D}^2_mE(\mu,x,y)\|_{\mathbf{op}}&\leq M^E_2;\\\underline{\lambda}&:=\inf_{\mu,x}\lambda_{\nabla\mathcal{D}_{m}H(\mu,x)+\lambda_\mu\nabla\mathcal{D}_{m}E(\mu,x)}>0;\\
    \|\nabla\frac{\delta\lambda}{\delta m}\|_{\infty}&<+\infty.
\end{align}
\end{assumption}

\begin{assumption}[\texttt{Log-Sobolev inequality}]\label{Assumption5}
We assume that there are $M^E_1\geq0$, $\rho_{LS}>0$ and $\mu_\infty$ equilibrium of \cref{gradflowPDE} such that for all $x\in\mathbb{R}^d$ and $\mu\in\mathcal{C}_M$,
\begin{align}
    \|\mathcal{D}_mE(\mu,x)\|_2&\leq M^E_1;\\
    2\rho_{LS}\bigg(\mathcal{E}(\mu)-\mathcal{E}(\mu_\infty)\bigg)&\leq\int\bigg\|\nabla\frac{\delta\mathcal{E}}{\delta m}(\mu,\cdot)\bigg\|^2_2d\mu.
\end{align}
\end{assumption}

\begin{assumption}[\texttt{Constrained log-Sobolev inequality}]\label{Assumption7}
We assume that there exists $\rho_{LS}>0$ and $\mu_\infty$ equilibrium of \cref{gradflowPDE} such that for all $\mu\in\mathcal{C}_M$,
\begin{align}
    2\rho_{LS}\bigg(\mathcal{E}(\mu)-\mathcal{E}(\mu_\infty)\bigg)\leq\int \bigg\|\nabla\bigg(\frac{\delta\mathcal{E}}{\delta m}(\mu,\cdot)+\lambda_\mu\frac{\delta E}{\delta m}(\mu,\cdot)\bigg)\bigg\|^2_2d\mu.
\end{align}
\end{assumption}

\begin{assumption}[\texttt{Constrained microscopic log-Sobolev inequality}]\label{Assumption8}
We assume that there exists $\rho_{LS}(\mu^n_\infty)>0$ and $\mu^n_\infty$ explicit equilibrium of \cref{microscopicflowPDE} such that
\begin{align}
    2\rho_{LS}(\mu^n_\infty)\mathbf{H}[\cdot|\mu^n_\infty]\leq\mathbf{I}[\cdot|\mu^n_\infty].
\end{align}
\end{assumption}

\subsection*{\underline{Analytical properties}}
We now summarize the fundamental analytical properties of the constrained gradient flow equation \cref{gradflowPDE}. These results show that the redistribution dynamics are mathematically well posed, preserve the feasibility constraint and dissipate the free energy monotonically over time. More precisely, we establish existence of weak solutions, invariance of the constraint manifold $\mathcal{C}_{M}$, an energy dissipation inequality and convergence towards a constrained equilibrium under displacement convexity. Proofs follow the standard minimizing‑movement framework for Wasserstein gradient flows
and are provided in \cref{ProofmainTheorems}.

\begin{theorem}[\texttt{Existence of weak solutions and invariance}]\label{Theorem1}
Under assumption \ref{Assumption1}, for all $\mu_{0}\in\mathcal{C}_{M}$ with finite entropy, there exists a global weak solution $(\mu_{t})_{t\geq0}\in\mathcal{C}([0,+\infty),\mathcal{P}_2(\mathbb{R}^d))$ of \cref{gradflowPDE} such that for all $t\geq0$, we have $\mu_t\geq0$, $\mu_{t}(\mathbb{R}^d)=1$ and $\mu_{t}\in\mathcal{C}_{M}$. Moreover the energy functional $t\mapsto\mathcal{E}(\mu_t)$ is non‑increasing and locally absolutely continuous.
\end{theorem}
The proof follows the constrained minimizing‑movement (JKO) scheme (a proximal gradient method),
\begin{equation}\label{SchmeJKO}
    \nu_{n+1}\in\textbf{argmin}_{\nu\in\mathcal{C}_{M}}\bigg\{\mathcal{E}(\nu)+\frac{1}{2\tau}\mathcal{W}^2_{2}(\nu,\nu_{n}) \bigg\}.
\end{equation}
Compactness and lower semi‑continuity yield convergence towards a curve of maximal slope in the metric space $(\mathcal{C}_{M},\mathcal{W}_2)$. Energy preservation follows by testing the weak formulation of equation \cref{gradflowPDE} with the function $x$. To derive first‑order conditions for \cref{SchmeJKO}, introduce the Lagrangian
\begin{equation}
    \mathcal{L}(\nu,\lambda):=\mathcal{E}(\nu)+\frac{1}{2\tau}\mathcal{W}^2_{2}(\nu,\nu_{n})+\lambda(E(\nu)-M),
\end{equation}
where $\lambda\in\mathbb{R}$ enforces the energy constraint. Let $\varphi_{n+1}$ denote a Kantorovich potential transporting $\nu_{n+1}$ to $\nu_{n}$. Then there exists $\lambda_{n+1}\in\mathbb{R}$ such that the Euler–Lagrange condition reads
\begin{equation}
    \frac{1}{\tau}\varphi_{n+1}(x)=\frac{\delta\mathcal{E}}{\delta m}(\nu_{n+1},x)+\lambda_{n+1}\frac{\delta E}{\delta m}(\nu_{n+1},x),\quad\textnormal{for a.e. }x\in\textbf{supp}(\nu_{n+1}).
\end{equation}
The additional term $\lambda_{n+1}\frac{\delta E}{\delta m}(\nu_{n+1},x)$ is the minimal correction (in the Wasserstein geometry) required to ensure feasibility of the update, i.e. to keep $\nu_{n+1}$ on $\mathcal{C}_M$.

\begin{theorem}[\texttt{Energy dissipation}]\label{Theorem2}
    Under assumptions \ref{Assumption1} and \ref{Assumption2}, for all  weak solution $(\mu_{t})_{t\geq0}$ of \cref{gradflowPDE}, for every $T>0$, we have  the monotone decay of the free energy along the constrained evolution:
    \begin{equation}\label{EnergyDissipationInequality}
        \mathcal{E}(\mu_{T})+\frac{1}{2}\int_{0}^{T}\|\partial_{t}\mu_{t}\|^2_{\mathcal{W}_2}dt+\frac{1}{2}\int_{0}^{T}\int\bigg\|\nabla\bigg(\frac{\delta\mathcal{E}}{\delta m}(\mu_{t},x)+\lambda_t\frac{\delta E}{\delta m}(\mu_t,x)\bigg)\bigg\|^2_{2}\mu_{t}(dx)dt\leq\mathcal{E}(\mu_0).
    \end{equation}
    Equality holds if and only if $(\mu_{t})_{t\geq0}$ follows a geodesic of steepest descent in the metric space $(\mathcal{C}_{M},\mathcal{W}_2)$.
\end{theorem}

\begin{theorem}[\texttt{Exponential convergence}]\label{Theorem3}
    Under assumption \ref{Assumption1}, if the submanifold $\mathcal{C}_M$ is a geodesic subspace of the Wasserstein space $(\mathcal{P}_2(\mathbb{R}^d),\mathcal{W}_2)$ and $\mathcal{E}$ is $\lambda$-geodesically convex on $\mathcal{C}_M$ for some  $\lambda>0$, the constrained gradient flow \cref{gradflowPDE} converges exponentially to the unique equilibrium $\mu_{\infty}\in\mathcal{C}_M$ (defined in \cref{Equilibria}):
    \begin{align}
        \mathcal{W}_2(\mu_t,\mu_\infty)&\leq e^{-\lambda t}\mathcal{W}(\mu_0,\mu_\infty);\\
        \mathcal{E}(\mu_t)-\mathcal{E}(\mu_\infty)&\leq \bigg(\mathcal{E}(\mu_0)-\mathcal{E}(\mu_\infty)\bigg)e^{-2\lambda t}.
    \end{align}
\end{theorem}
The argument follows the evolution variational inequality framework for Wasserstein gradient flows \cite{AmbrosioGigliSavare}. For every $\nu\in\mathcal{C}_M$,
\begin{equation}
    \frac{1}{2}\frac{d}{dt}\mathcal{W}^2_{2}(\mu_t,\nu)+\frac{\lambda}{2}\mathcal{W}^2_{2}(\mu_t,\nu)\leq\mathcal{E}(\nu)-\mathcal{E}(\mu_t).
\end{equation}
Choosing $\nu=\mu_{\infty}$ and applying Grönwall’s inequality yields the exponential contraction in both Wasserstein distance and energy. 

\subsection*{\underline{Particle approximation}}
Consider the particle system given by the $X^{n,i}_t$ defined by \cref{ConstrainedParticleSystem} and synchronously coupled to $n$ copies $X^i_t$ of $X_t$ defined by \cref{ConstrainedMeanFieldLimit}. 
\begin{theorem}[\texttt{Propagation of chaos and Lipschitz regularity of the first flat derivatives}]\label{Theorem4}
  Under assumption \ref{Assumption3}, \cref{TheoremBoundLambda} and $|\lambda^n_t|\leq C_{\lambda}(t)$ almost surely, we have 
  \begin{align}
      \mathbb{E}[\|X^{n,i}_t-X^i_t\|^2_2]&\leq C(d)\delta_d(n)\int_{0}^{t}\bigg(M^H+C_\lambda(s)M^E+C_{\Delta\nabla}M^E_1\bigg)M(\mu_s)e^{\int_{s}^{t}\bigg(4M^H+4C_{\lambda}(u)M^E+3C_{\Delta\nabla}M^E_1\bigg)du}ds;\\
      \mathbb{E}[|\lambda^{n}_t-\lambda_t|^2]&\leq 2C_{\Delta\nabla}^2\bigg(\mathbb{E}[\|X^{n,i}_t-X^i_t\|^2_2]+C(d)M(\mu_t)\delta_d(n)\bigg);
  \end{align}
  where $C(d)M(\mu_s)\delta_d(n)$ is a Fournier-Guillin bound:
  \begin{equation}
      \delta_d(n):=\begin{cases}
          \frac{1}{\sqrt{n}}\quad&\textnormal{if } d<4;\\
          \frac{\ln(n+1)}{\sqrt{n}}\quad&\textnormal{if } d=4;\\
          \frac{1}{n^{\frac{2}{d}}}\quad&\textnormal{if } d>4.
      \end{cases}
  \end{equation}
\end{theorem}

\begin{remark}
    For some examples of constant $C_{\Delta\nabla}$ from \cref{TheoremBoundLambda} and the  almost-sure deterministic upper bound $C_\lambda(t)$ for the Lagrange multipliers, see \cref{ExempleAppThmBoundLambda} and \cref{FlatDerivativeLambda}. As for the moments $M(\mu_t)$, see \cref{TheoremMoments}.
\end{remark}

\begin{theorem}[\texttt{Propagation of chaos and flat second derivatives}]\label{Theorem5}
    Under assumption \ref{Assumption4}, \cref{TheoremBoundLambda} and $|\lambda^n_t|\leq C_{\lambda}(t)$ almost surely, we have
    \begin{align}
    \mathbb{E}[\|X^{n,i}_t-X^i_t\|^2_2]&\leq C(d)\delta_d(n)\int_{0}^{t}\bigg(M^H_2+C_\lambda(s)M^E+C_{\Delta\nabla}M^E_1\bigg)M(\mu_s)e^{\int_{s}^{t}\bigg(3M^H_2+5C_\lambda(u)M^E+3C_{\Delta\nabla}M^E_1-2\lambda_H\bigg)du}ds;\\
    \mathbb{E}[|\lambda^{n}_t-\lambda_t|^2]&\leq 2C_{\Delta\nabla}^2\bigg(\mathbb{E}[\|X^{n,i}_t-X^i_t\|^2_2]+C(d)M(\mu_t)\delta_d(n)\bigg).
\end{align}  
Moreover, if $M_\infty:=\sup_{t\geq0}M(\mu_t)<+\infty$ and there exists $C_\lambda\geq0$ such that $\omega:=2\lambda_H-3M^H_2-5C_\lambda M^E-3C_{\Delta\nabla}M^E_1>0$ and $C_\lambda(t)\leq C_\lambda$, this upper bound is uniform in time:
\begin{align}
    \sup_{t\geq0}\mathbb{E}[\|X^{n,i}_t-X^i_t\|^2_2]&\leq M_\infty C(d)\delta_d(n)\bigg(M^H_2+C_\lambda M^E+C_{\Delta\nabla}M^E_1\bigg)\frac{1-e^{-\omega t}}{\omega};\\
    \sup_{t\geq0}\mathbb{E}[|\lambda^{n}_t-\lambda_t|^2]&\leq 2C_{\Delta\nabla}^2\bigg(\sup_{t\geq0}\mathbb{E}[\|X^{n,i}_t-X^i_t\|^2_2]+M_\infty C(d)\delta_d(n)\bigg).
\end{align}
\end{theorem}
\begin{remark}
If the Hamiltonian $H$ is not dissipative but its Lipschitz intrinsic derivative in its two arguments and the energy $E$ is dissipative and its first and second intrinsic derivatives are bounded, we can redo the proof of \cref{Theorem5}.
\end{remark}

\subsection*{\underline{Contraction via synchronous coupling}}
Let $(X_t,\mu_t,\lambda_t)_{\geq0}$ and $(Y_t,\nu_t,\beta_t)_{\geq0}$ two solutions of \cref{ConstrainedMeanFieldLimit}.

\begin{theorem}[\texttt{Contraction by synchronous coupling}]\label{Theorem6}
Under assumption \ref{Assumption3} and \cref{TheoremBoundLambda}, we have
    \begin{align}
    \mathcal{W}_2(\mu_t,\nu_t)&\leq \mathcal{W}_2(\mu_0,\nu_0)e^{\int_0^t(2M^H+2|\beta_s|M^E+C_{\Delta\nabla}M^E_1)ds};\\
    |\lambda_t-\beta_t|&\leq C_{\Delta\nabla}\mathcal{W}_2(\mu_t,\nu_t).
\end{align}
Moreover, for all equilibrium $\mu_{\infty}$ of \cref{gradflowPDE}, we have
\begin{align}
    \mathcal{W}_2(\mu_t,\mu_\infty)&\leq \mathcal{W}_2(\mu_0,\mu_\infty)e^{\omega_{\infty}t};\\
    |\lambda_t-\lambda_\infty|&\leq C_{\Delta\nabla}\mathcal{W}_2(\mu_t,\mu_\infty);
\end{align}
with $\omega_{\infty}:=2M^H+2|\lambda_\infty|M^E+C_{\Delta\nabla}M^E_1$.
\end{theorem}

\begin{theorem}[\texttt{Contraction by synchronous coupling and dissipativity}]\label{Theorem7}
Under assumption \ref{Assumption4} and \cref{TheoremBoundLambda}, we have
\begin{align}
    \mathcal{W}_2(\mu_t,\nu_t)&\leq \mathcal{W}_2(\mu_0,\nu_0)e^{\int_0^t(M^H_2+2|\beta_s|M^E+C_{\Delta\nabla}M^E_1-\lambda_H)ds};\\
    |\lambda_t-\beta_t|&\leq C_{\Delta\nabla}\mathcal{W}_2(\mu_t,\nu_t).
\end{align}
Moreover, for all equilibrium $\mu_{\infty}$ of \cref{gradflowPDE}, we have
\begin{align}
    \mathcal{W}_2(\mu_t,\mu_\infty)&\leq \mathcal{W}_2(\mu_0,\mu_\infty)e^{-\omega_{\infty}t};\\
    |\lambda_t-\lambda_\infty|&\leq C_{\Delta\nabla}\mathcal{W}_2(\mu_t,\mu_\infty);
\end{align}
with $\omega_{\infty}:=\lambda_H-M^H_2-2|\lambda_\infty|M^E-C_{\Delta\nabla}M^E_1$. In particular, if $\omega_\infty>0$, there is a unique equilibrium $\mu_\infty$ and for all initial condition $\mu_0$, $\mathcal{W}_2(\mu_t,\mu_\infty)\overset{t\to+\infty}{\longrightarrow}0$ and $|\lambda_t-\lambda_\infty|\overset{t\to+\infty}{\longrightarrow}0$ exponentially.
\end{theorem}
\begin{remark}If the Hamiltonian $H$ is not dissipative but its Lipschitz intrinsic derivative in its two arguments and the energy $E$ is dissipative and its first and second intrinsic derivatives are bounded, we can redo the proof of \cref{Theorem7}.
\end{remark}
\begin{remark}
The \cref{Theorem4}, \cref{Theorem5}, \cref{Theorem6} and \cref{Theorem7} can be obtained using general drifts that are not necessarily intrinsic derivatives; without loss of generality, the proofs are carried out using intrinsic derivatives.
\end{remark}
Let
\begin{equation}
    \Phi(\mu)(dx):=\frac{1}{Z_\mu}e^{-\frac{1}{\sigma^2}\bigg(\frac{\delta H}{\delta m}(\mu,x)+\lambda_{\mu}\frac{\delta E}{\delta m}(\mu,x)\bigg)}dx.
\end{equation}

\begin{theorem}[\texttt{Contraction of the invariant measure map}]\label{Theorem8}
Under assumption \ref{Assumption6}, we have
\begin{equation}
    \mathcal{W}_1(\Phi(\nu_0),\Phi(\nu_1))\leq\frac{M^H_2+C_\lambda M^E_2+\|\nabla\frac{\delta\lambda}{\delta m}\|_{\infty}M^E_1}{\sigma^2\underline{\lambda}}\mathcal{W}_1(\nu_0,\nu_1).
\end{equation}
In particular, if
\begin{equation}
    \frac{M^H_2+C_\lambda M^E_2+\|\nabla\frac{\delta\lambda}{\delta m}\|_{\infty}M^E_1}{\sigma^2\underline{\lambda}}<1,
\end{equation}
there exists at most one invariant measure. Moreover, if $\mathcal{C}_M$ is a closed subset of $\mathcal{P}_1(\mathbb{R}^d)$ and $\Phi(\mathcal{C}_M)\subset\mathcal{C}_M$, there exists a unique invariant measure.
\end{theorem}

\subsection*{\underline{Ergodicity via functional inequality}}

\begin{theorem}[\texttt{Entropy exponential convergence by log-Sobolev inequality}]\label{Theorem9}
Under assumption \ref{Assumption5}, we have
\begin{equation}
    \mathcal{E}(\mu_t)-\mathcal{E}(\mu_\infty)\leq\bigg(\mathcal{E}(\mu_0)-\mathcal{E}(\mu_\infty)\bigg)e^{-\rho_{LS}t}+\frac{(M^E_1)^2}{2}\int_0^t|\lambda_s|^2e^{-\rho_{LS}(t-s)}ds.
\end{equation}
If the geometry of the manifold $\mathcal{C}_M$ ensures the Otto-Villani implication, from this exponential bound of the mean field entropy $\mathcal{E}-\mathcal{E}(\mu_\infty)$ along the flow $(\mu_t)_{t\geq0}$, we deduce an exponential bound in Wasserstein metric:
\begin{equation}
    \mathcal{W}^2_2(\mu_t,\mu_\infty)\leq\frac{2}{\rho_{LS}}\bigg(\mathcal{E}(\mu_t)-\mathcal{E}(\mu_\infty)\bigg)
\end{equation}
with $\rho_{LS}$ potentially modified by that geometry.
\end{theorem}

\begin{theorem}[\texttt{Entropy exponential convergence by constrained log-Sobolev inequality}]\label{Theorem10}
Under assumption \ref{Assumption7}, we have
\begin{equation}
    \mathcal{E}(\mu_t)-\mathcal{E}(\mu_\infty)\leq\bigg(\mathcal{E}(\mu_0)-\mathcal{E}(\mu_\infty)\bigg)e^{-2\rho_{LS}t}.
\end{equation}
If the geometry of the manifold $\mathcal{C}_M$ ensures the Otto-Villani implication, from this exponential bound of the mean field entropy $\mathcal{E}-\mathcal{E}(\mu_\infty)$ along the flow $(\mu_t)_{t\geq0}$, we deduce an exponential bound in Wasserstein metric:
\begin{equation}
    \mathcal{W}^2_2(\mu_t,\mu_\infty)\leq\frac{2}{\rho_{LS}}\bigg(\mathcal{E}(\mu_t)-\mathcal{E}(\mu_\infty)\bigg)
\end{equation}
with $\rho_{LS}$ potentially modified by that geometry.  
\end{theorem}
\begin{remark}
The minimal condition to ensure the Otto-Villani implication in \cite{OttoVillani} is that the functional $E$ must be displacement-linear (defining a totally geodesic subspace) or have a controlled second fundamental form (induced curvature).
\end{remark}

\begin{theorem}[\texttt{Entropy contraction by microscopic constrained log-Sobolev inequality}]\label{Theorem11}
Under assumption \ref{Assumption8}, we have
  \begin{equation}
      \mathbf{H}[\mu^n_t|\mu^n_\infty]\leq e^{-2\rho_{LS}(\mu^n_\infty)t}\mathbf{H}[\mu^n_0|\mu^n_\infty].
  \end{equation}
  From this exponential bound of the microscopic entropy along the microscopic flow $(\mu^n_t)_{t\geq0}$, we deduce an exponential bound in Wasserstein metric:
  \begin{equation}
      \mathcal{W}^2_2(\mu^n_t,\mu^n_\infty)\leq\frac{2}{\rho_{LS}(\mu^n_\infty)}\mathbf{H}[\mu^n_t|\mu^n_\infty].
  \end{equation}
\end{theorem}

\section{Characterization of stationary states}\label{Equilibria}
\begin{proposition}[\texttt{Constrained equilibria}]\label{PropositionEquilibria}
    Under assumption \ref{Assumption1}, a stationary solution $\mu_\infty$ of \cref{gradflowPDE} satisfies, on its support,
    \begin{equation}\label{EquilibriumEquation}
        \frac{\delta\mathcal{E}}{\delta m}(\mu_{\infty},x)+\lambda_{\infty}\frac{\delta E}{\delta m}(\mu_{\infty},x)=C
    \end{equation}
    for some constants $C,\lambda_\infty\in\mathbb{R}$. Equivalently,
    \begin{equation}
        \frac{\delta H}{\delta m}(\mu_{\infty},x)+\sigma^2\ln\frac{d\mu_{\infty}}{dx}+\lambda_{\infty}\frac{\delta E}{\delta m}(\mu_{\infty},x)=C\Longleftrightarrow\mu_\infty(dx)=\frac{1}{Z_\infty}e^{-\frac{1}{\sigma^2}\bigg(\frac{\delta H}{\delta m}(\mu_{\infty},x)+\lambda_{\infty}\frac{\delta E}{\delta m}(\mu_{\infty},x)\bigg)}dx.
    \end{equation}
    The multiplier $\lambda_\infty$ is uniquely determined by the constraint $E(\mu_\infty)=M$.
\end{proposition}
\begin{proof}
    If $\mu_\infty$ is stationary, then $\partial_t\mu_\infty=0$  in \cref{gradflowPDE}, hence 
    \begin{equation}
        \nabla\cdot\bigg(\mu_{\infty}\nabla\frac{\delta}{\delta m}\bigg( H+\lambda_{\infty} E\bigg)(\mu_{\infty},\cdot)+\sigma^2\nabla\mu_{\infty}\bigg)=0,
    \end{equation}
    which implies
    \begin{equation}
        \nabla\frac{\delta\mathcal{E}}{\delta m}(\mu_{\infty},\cdot)+\lambda_{\infty}\nabla\frac{\delta E}{\delta m}(\mu_{\infty},\cdot)\equiv0\quad\textnormal{on }\textbf{supp}(\mu_\infty).
    \end{equation}
    From this, we derive \cref{EquilibriumEquation}. Conversely, any $\mu_\infty$ satisfying equation \cref{EquilibriumEquation} and $E(\mu_\infty)=M$ is stationary for \cref{gradflowPDE}. 
\end{proof}
\cref{EquilibriumEquation} is the Euler–Lagrange condition for the constrained minimization of $\mathcal{E}$ on $\mathcal{C}_M$. It characterizes equilibrium distributions of the constrained gradient flow and generalizes the classical Gibbs measure associated with unconstrained mean-field Langevin dynamics.

\section{Particle approximation}\label{SectionParticleApproximation}
We define the system of particles with constrained mean field associated with \cref{gradflowPDE}:
\begin{equation}\label{ConstrainedParticleSystem}
    \forall n\geq 1\quad\forall i\in\{1,\ldots n\}\quad\forall t\geq0,\quad dX^{n,i}_t=\sigma\sqrt{2}dB^{i}_t-\nabla\frac{\delta H}{\delta m}(\mu_{\mathbf{X}_t},X^{n,i}_t)dt-\lambda^n_t\nabla\frac{\delta E}{\delta m}(\mu_{\mathbf{X}_t},X^{n,i}_t)dt
\end{equation}
where $\mu_{\mathbf{X}_t}:=\frac{1}{n}\sum_{i=1}^{n}\delta_{X^{n,i}_t}$ and the multiplier $\lambda^n_t$ enforces the constraint $E(\mu_{\mathbf{X}_t})=M$ (see \cref{Thmexplicitformlambda}):
\begin{align}
\lambda^n_t&:=\lambda_{\mu_{\mathbf{X}_t}}= -\frac{1}{\| \nabla \frac{\delta E}{\delta m}(\mu_{\mathbf{X}_t},\cdot)\|^2_{L^2(\mu_{\mathbf{X}_t})}}\int_{\mathbb{R}^d} \nabla \frac{\delta \mathcal{E}}{\delta m}(\mu_{\mathbf{X}_t},x)\cdot\nabla \frac{\delta E}{\delta m}(\mu_{\mathbf{X}_t},x) \mu_{\mathbf{X}_t}(dx)\\
&=\frac{1}{\frac{1}{n}\sum_{j=1}^n\| \nabla \frac{\delta E}{\delta m}(\mu_{\mathbf{X}_t},X^{n,j}_t)\|^2_2}\bigg(\frac{\sigma^2}{n}\sum_{j=1}^{n}\Delta\frac{\delta E}{\delta m}(\mu_{\mathbf{X}_t},X^{n,j}_t)-\frac{1}{n}\sum_{j=1}^n\nabla \frac{\delta H}{\delta m}(\mu_{\mathbf{X}_t},X^{n,j}_t)\cdot\nabla \frac{\delta E}{\delta m}(\mu_{\mathbf{X}_t},X^{n,j}_t)\bigg)\nonumber\\
&=:\xi_n(\mathbf{X}_t).
\end{align}
We have the linear constrained Langevin dynamics:
\begin{equation}\label{ConstrainedLangevinSystem}
    d\mathbf{X}_t=\sigma\sqrt{2}d\mathbf{B}_t-\nabla H_n(\mathbf{X}_t)dt-\xi_n(\mathbf{X}_t)\nabla E_n(\mathbf{X}_t)dt
\end{equation}
with
\begin{align}\label{ConstrainedMicroscopicHamiltonian}
    H_n(\mathbf{x})&:=nH(\mu_{\mathbf{x}});\\
    E_n(\mathbf{x})&:=nE(\mu_{\mathbf{x}});
\end{align}
the constrained microscopic Hamiltonians. By noting that
\begin{align}
    \nabla_{x_i}H_n(\mathbf{x})&=\nabla\frac{\delta H}{\delta m}(\mu_{\mathbf{x}},x_i);\\
    \nabla_{x_i}E_n(\mathbf{x})&=\nabla\frac{\delta E}{\delta m}(\mu_{\mathbf{x}},x_i);
\end{align}
we have
\begin{equation}\label{multiplierEmpiricfunction}
    \xi_n=\frac{1}{\frac{1}{n}\sum_{j=1}^{n}\|\nabla_{x_j}E_n\|^2_2}\bigg(\frac{\sigma^2}{n}\sum_{j=1}^n\Delta_{x_j}E_n-\frac{1}{n}\sum_{j=1}^n\nabla_{x_j}H_n\cdot\nabla_{x_j}E_n\bigg).
\end{equation}
The law $\mu^n_t:=\mathbb{P}_{\mathbf{X}_t}$ satisfies the linear constrained Fokker-Planck equation:
\begin{equation}\label{microscopicflowPDE}
\partial_t\mu^n_t=\nabla\cdot\bigg(\mu^n_t(\nabla H_n+\xi_n\nabla E_n)+\sigma^2\nabla\mu^n_t\bigg).
\end{equation}
A microscopic equilibrium $\mu^n_\infty$ is given by
\begin{equation}
    \nabla\cdot\bigg(\mu^n_\infty(\nabla H_n+\xi_n\nabla E_n)+\sigma^2\nabla\mu^n_\infty\bigg)=0.
\end{equation}
\begin{theorem}[\texttt{Explicit constrained microscopic equilibrium}]\label{ThmExplicitMicroscopicIM}
    \cref{ConstrainedLangevinSystem} admits an explicit invariant measure if and only if there exists a differentiable scalar-valued function $\psi_n$ such that $\xi_n=\psi_n'(E_n)$. In this case, we have
    \begin{equation}\label{ExplicitMicroscopicIM}
    \mu^n_\infty(d\mathbf{x}):=\frac{1}{Z_n}e^{-\frac{1}{\sigma^2}\bigg(H_n(\mathbf{x})+\psi_n(E_n(\mathbf{x}))\bigg)}d\mathbf{x}.
\end{equation}

\end{theorem}
\begin{proof}
For overdamped Langevin dynamics of the form of \cref{ConstrainedLangevinSystem}, an explicit invariant measure exists if and only if the system is reversible (conservative):
    \begin{equation}
        \xi_n=\psi_n'(E_n).
    \end{equation}
    In this case, the total drift derives entirely from an effective potential:
    \begin{equation}
        \nabla H_n+\xi_n\nabla E_n=\nabla(H_n+\psi_n(E_n))
    \end{equation}
    and we have
    \begin{equation}
    \mu^n_\infty(d\mathbf{x}):=\frac{1}{Z_n}e^{-\frac{1}{\sigma^2}\bigg(H_n(\mathbf{x})+\psi_n(E_n(\mathbf{x}))\bigg)}d\mathbf{x}.
\end{equation}
\end{proof}

\subsection*{\underline{Constrained mean-field limit}}
The associated mean-field limit is defined as:
\begin{equation}\label{ConstrainedMeanFieldLimit}
    \forall t\geq0,\quad dX_t=\sigma\sqrt{2}dB_t-\nabla\frac{\delta H}{\delta m}(\mu_t,X_t)dt-\lambda_t\nabla\frac{\delta E}{\delta m}(\mu_t,X_t)dt.
\end{equation}
We recall that (see \cref{Thmexplicitformlambda})
\begin{align}
\lambda_t&=\lambda_{\mu_t}= -\frac{1}{\| \nabla \frac{\delta E}{\delta m}(\mu_t,\cdot)\|^2_{L^2(\mu_t)}}\int_{\mathbb{R}^d} \nabla \frac{\delta \mathcal{E}}{\delta m}(\mu_t,x)\cdot\nabla \frac{\delta E}{\delta m}(\mu_t,x) \mu_t(dx)\\
&=\frac{1}{\| \nabla \frac{\delta E}{\delta m}(\mu_t,\cdot)\|^2_{L^2(\mu_t)}}\bigg(\sigma^2\int_{\mathbb{R}^d}\Delta\frac{\delta E}{\delta m}(\mu_t,\cdot)d\mu_t-\int_{\mathbb{R}^d} \nabla \frac{\delta H}{\delta m}(\mu_t,x)\cdot\nabla \frac{\delta E}{\delta m}(\mu_t,x) \mu_t(dx)\bigg).\nonumber
\end{align}
To demonstrate that this is indeed the mean-field limit, we consider a synchronous coupling between the particle system and the latter: we show the propagation of chaos. To achieve this, the Lagrange multiplier must be Lipschitz in its measure argument with respect to the Wasserstein metric, and the empirical Lagrange multiplier $\lambda^n_t$ must admits, almost surely, a deterministic upper bound $C_\lambda(t)$ independent of the number of particles.
\begin{theorem}[\texttt{Lipschitz Lagrange multiplier}]\label{TheoremBoundLambda}
Assume that there exists $C_{\Delta\nabla}\geq0$ such that for all $\mu\in\mathcal{C}_M$ and $\nu\in\mathcal{C}_M$, we have
    \begin{equation}
        |\lambda_\mu-\lambda_\nu|\leq C_{\Delta\nabla}\mathcal{W}_2(\mu,\nu).
    \end{equation}
    It follows that for two solutions $(X_t,\mu_t,\lambda_t)_{\geq0}$ and $(Y_t,\nu_t,\beta_t)_{\geq0}$ of \cref{ConstrainedMeanFieldLimit}, we have the following bounds on the multipliers:
    \begin{align}
        |\lambda_t-\beta_t|^2&\leq C_{\Delta\nabla}^2\mathbb{E}[\|X_t-Y_t\|^2_2];\\
        \mathbb{E}[|\lambda^{n}_t-\lambda_t|^2]&\leq 2C_{\Delta\nabla}^2\bigg(\mathbb{E}[\|X^{n,i}_t-X^i_t\|^2_2]+C(d)M(\mu_t)\delta_d(n)\bigg).
    \end{align}
\end{theorem}
\begin{proof}
    By definition of Wasserstein metric, triangular and Young inequalities, and Fournier-Guillin bound, we have
    \begin{align}
        |\lambda_t-\beta_t|&\leq C_{\Delta\nabla}\mathcal{W}_2(\mu_t,\nu_t)\\
        &\leq C_{\Delta\nabla}\sqrt{\mathbb{E}[\|X_t-Y_t\|^2_2]}\nonumber;\\
        |\lambda^n_t-\lambda_t|^2&\leq 2C_{\Delta\nabla}^2\bigg(\mathcal{W}^2_2(\mu_{\mathbf{X}_t},\mu_{\overline{\mathbf{X}}_t})+\mathcal{W}^2_2(\mu_{\overline{\mathbf{X}}_t},\mu_t)\bigg);\\
        \mathbb{E}[|\lambda^n_t-\lambda_t|^2]&\leq 2C_{\Delta\nabla}^2\bigg(\mathbb{E}[\|X^{n,i}_t-X^i_t\|^2_2]+C(d)M(\mu_t)\delta_d(n)\bigg).
    \end{align}
\end{proof}
As for the control of moments $M(\mu_t)$, we have the following result.
\begin{theorem}[\texttt{Bounds on the moments}]\label{TheoremMoments}
    We have
    \begin{equation}
        M(\mu_t)\leq\bigg(M(\mu_0)+\mathcal{E}(\mu_0)-\int_0^t\mathcal{E}(\mu_s)e^{-s}ds\bigg)e^t-\mathcal{E}(\mu_t).
    \end{equation}
    In particular, if $M(\mu_0)<+\infty$ and $\mathcal{E}(\mu_0)<+\infty$, for all $T\geq0$,
    \begin{equation}
        \sup_{t\leq T}M(\mu_t)<+\infty.
    \end{equation}
    Under assumption \ref{Assumption4}, $M_0:=\sup_{\mu}\|\nabla\frac{\delta H}{\delta m}(\mu,0)\|_2<+\infty$ and $\omega:=2\lambda_H-2>0$, we have
    \begin{equation}
        M(\mu_t)\leq M(\mu_0)e^{-\omega t}+\bigg(M_0+(C_\lambda M^{E}_1)^2\bigg)\frac{1-e^{-\omega t}}{\omega}\quad (M_\infty:=\sup_{t\geq0}M(\mu_t)<+\infty).
    \end{equation}
\end{theorem}
\begin{proof}
    By Cauchy-Schwarz and Young inequalities, and \cref{PropoEntropyProd}, we have 
    \begin{align}
        \frac{d}{dt}M(\mu_t)&:=\frac{d}{dt}\int\|x\|^2_2\mu_t(dx)=-2\int x\cdot\nabla\bigg(\frac{\delta\mathcal{E}}{\delta m}(\mu_t,x)+\lambda_t\frac{\delta E}{\delta m}(\mu_t,x)\bigg)\mu_t(dx)\\
        &\leq\int\|x\|^2_2\mu_t(dx)+\int \bigg\|\nabla\bigg(\frac{\delta\mathcal{E}}{\delta m}(\mu_t,\cdot)+\lambda_t\frac{\delta E}{\delta m}(\mu_t,\cdot)\bigg)\bigg\|^2_2d\mu_t\nonumber\\
        &= M(\mu_t)-\frac{d}{dt}\mathcal{E}(\mu_t).\nonumber
    \end{align}
    By Gronwall's lemma, we have
\begin{equation}
    M(\mu_t)\leq M(\mu_0)e^{t}-\int_{0}^t\frac{d}{ds}\mathcal{E}(\mu_s)e^{t-s}ds.
\end{equation}
By integration by parts, we have
\begin{equation}
    \int_{0}^t\frac{d}{ds}\mathcal{E}(\mu_s)e^{t-s}ds=[\mathcal{E}(\mu_s)e^{t-s}]_0^t+\int_0^t\mathcal{E}(\mu_s)e^{t-s}ds=\mathcal{E}(\mu_t)-\mathcal{E}(\mu_0)e^{t}+e^t\int_0^t\mathcal{E}(\mu_s)e^{-s}ds.
\end{equation}
The second inequality follows from triangular inequality and \cref{PropoFreeEnergyLyapunov} (or \cref{PropoEntropyProd}). By dissipativity condition of assumption \ref{Assumption4}, the divergence formula, \cref{Thmexplicitformlambda} and Cauchy-Schwarz inequality, we have
\begin{align}
    \frac{d}{dt}M(\mu_t)&=-2\int x\cdot\nabla\bigg(\frac{\delta H}{\delta m}(\mu_t,x)+\lambda_t\frac{\delta E}{\delta m}(\mu_t,x)\bigg)\mu_t(dx)\\
    &\leq-\omega M(\mu_t)+M_0+(C_\lambda M^{E}_1)^2.\nonumber
\end{align}
By Gronwall's lemma, we have
\begin{equation}
    M_t\leq M(\mu_0) e^{-\omega t}+\bigg(M_0+(C_\lambda M^{E}_1)^2\bigg)\int_0^{t}e^{-\omega(t-s)}ds.
\end{equation}
\end{proof}

\begin{theorem}[\texttt{Corollary of \cref{Theorem4} or \cref{Theorem5}}]\label{ConseqChaosW2}
    The result of \cref{Theorem4} (or \cref{Theorem5}) ensure with explicit rates,
\begin{enumerate}
    \item the weak convergence of the law $\mu^{(n),p}_{t}$ of a particle towards $\mu_{t}$: in fact, $X^{(n),p} _{t}$ has the law $\mu^{(n),p}_{t}$ by definition and $X^{(p)}_{t}$ has the law $\mu_{t}$ by construction, therefore for all $t\leq T$ and all $n\geq1$,
    \begin{align}
        \mathcal{W}^{2}_{2}(\mu^{(n),p}_{t},\mu_{t})\leq\frac{M}{\sqrt{n}}
    \end{align}
    with $M$ not depending on the number of particles. The uniformity in time of $M$ is verified if $\omega>0$ in \cref{Theorem5}.
    \item the propagation of chaos for the particle system: $q$ being a fixed integer- or more generally a $\mathbf{o}(\sqrt{n})$- and for the Wasserstein distance defined on $(\mathbb{R} ^{d})^{q}$, for $\{i_{1},\ldots,i_{q}\}$ a part of $\{1,\ldots,n\}$, we have
    \begin{align}
        \mathcal{W}^{2}_{2}\bigg(\mathbb{P}_{(X^{(n),i_{1}}_{t},\ldots,X^{(n) ,i_{q}}_{t})},\mu^{\otimes q}_{t}\bigg)\leq\frac{qM}{\sqrt{n}};
    \end{align}
    \item the convergence of the empirical measurement of the particle system towards the McKean-Vlasov particle law: for any Lipschitzian function $\varphi: \mathbb{R}^{d}\to\mathbb{R}$, we have
    \begin{align}
        \mathbb{E}|\langle\Pi^{(n)}_{t},\varphi\rangle-\langle\mu_{t},\varphi\rangle|^2\leq2\frac{M+M_{2}} {\sqrt{n}}[\varphi]^2_{1}\quad\textnormal{with }M_{2}:=\sup_{t\leq T}\mathbb{E}[\|X_{t}\|^2].
    \end{align}
\end{enumerate}
\end{theorem}
\begin{proof}
By definition of the Wasserstein metric and empirical measurement, we have
\begin{enumerate}
    \item 
    $\mathcal{W}^{2}_{2}(\mu^{(n),p}_{t},\mu_{t})\leq\mathbb{E}||X^{(n) ,p}_{t}-X^{(p)}_{t}||^2\leq\frac{M}{\sqrt{n}}$;
    \item 
    $\mathcal{W}^{2}_{2}(\mathbb{P}_{(X^{(n),i_{1}}_{t},\ldots,X^{(n) ,i_{q}}_{t})},\mu^{\otimes q}_{t})\leq\mathbb{E}\|(X^{(n),i_ {1}}_{t},\ldots,X^{(n),i_{q}}_{t})-(X^{(i_{1})},\ldots,X^{(i_ {q})}_{t})\|^2\leq q\mathbb{E}||X^{(n) ,p}_{t}-X^{(p)}_{t}||^2\leq\frac{qM}{\sqrt{n}}=M\mathbf{o}(1)$;
    \item
    $\mathbb{E}|\langle\Pi^{(n)}_{t},\varphi\rangle-\langle\mu_{t},\varphi\rangle|^2\leq 2\mathbb{E}|\varphi(X^{(n),p}_{t})-\varphi(X^{(p)}_{t})|^2+2\mathbb{E}|\frac{1}{n}\sum_{i=1}^{n}\varphi(X^{(i)}_{t})-\langle\mu_{t},\varphi\rangle|^2$. It follows that $\mathbb{E}|\langle\Pi^{(n)}_{t},\varphi\rangle-\langle\mu_{t},\varphi\rangle|^2\leq2[\varphi]^{2}_{1}\frac{M}{\sqrt{n}}+2\mathbb{V}[\frac{1}{n}\sum_{i=1}^{n}\varphi(X^{(i)}_{t})]=2[\varphi]^{2}_{1}\frac{M}{\sqrt{n}}+\frac{2}{n}\mathbb{V}[\varphi(X^{(p)}_{t})]\leq2[\varphi]^{2}_{1}\frac{M}{\sqrt{n}}+\frac{2}{\sqrt{n}}\mathbb{V}[\varphi(X^{(p)}_{t})]$. And by applying the law of large numbers, we deduce that $\mathbb{V}[\varphi(X^{(p)}_{t})]\leq[\varphi]^{2}_{1}M_{2}$.
\end{enumerate}
\end{proof}

\subsection*{\underline{Some explicit conditions}}
We provide some explicit conditions ensuring the Lipschitz property of the multiplier and the existence of the almost sure deterministic bound. First, we present general properties concerning flat derivatives. Second, assuming the existence of second-order flat derivatives with suitable properties, we show that the multiplier admits flat derivatives which we explicitly characterize and from which we derive the desired results.\newline

In the constrained particle system and its mean-field limit, the empirical and theoretical Lagrange multipliers involve the average macroscopic quantity: 
\begin{equation}
    \mu\longmapsto\int_{\mathbb{R}^d}\Delta\frac{\delta E}{\delta m}(\mu,\cdot)d\mu.
\end{equation}
To ensure that the gradient flow in Wasserstein space does not develop singularities (instantaneous mass concentration or noise explosion), this functional must be bounded and locally Lipschitz continuous with respect to the $\mathcal{W}_2$ metric. This requires controlling not only the spatial second derivative of the flat derivative but also its variations with respect to the measure. To control this term in full generality, the constraint functional $E$ must satisfy the following higher regularity class:
\begin{itemize}
    \item \textbf{Spatial regularity of the Hessian:} 
The map $x \mapsto \frac{\delta E}{\delta m}(\mu,x)$ must be of class $\mathcal{C}^2$ in space. Moreover, its spatial Hessian must be globally bounded uniformly with respect to the measurement: 
\begin{equation} 
\sup_{\mu \in \mathcal{P}_2(\mathbb{R}^d)} \| \nabla^2  \frac{\delta E}{\delta m}(\mu,\cdot)\|_{\infty} \le M_{E} < +\infty .
\end{equation} 
Since $\Delta= \textbf{Tr}(\nabla^2)$, this condition immediately implies uniform control in space: $| \Delta \frac{\delta E}{\delta m}(\mu,x)| \le d \cdot M_{E}$.
\item \textbf{Lipschitzianity of the flat derivative of order 2:} 
To stabilize the fluctuations of the multiplier $|\lambda^n_t - \lambda_t|$, the spatial Laplacian must be Lipschitz-continuous with respect to the Wasserstein topology: 
\begin{equation} 
\left| \int_{\mathbb{R}^d} \Delta \frac{\delta E}{\delta m}(\mu,\cdot) d\mu - \int_{\mathbb{R}^d} \Delta \frac{\delta E}{\delta m}(\nu,\cdot) d\nu \right| \le L_{\Delta} \cdot \mathcal{W}_2(\mu, \nu). 
\end{equation}
\item \textbf{Combined coercivity hypothesis (denominator control):}
Since the Laplacian term is of higher topological order, it must be dominated by the macroscopic kinetic energy (the denominator of $\lambda_\mu$). There must exist a constant $\alpha \in [0, 1[$ such that:
\begin{equation}
\left| \int_{\mathbb{R}^d} \Delta \frac{\delta E}{\delta m}(\mu,\cdot) d\mu \right| \le \alpha \int_{\mathbb{R}^d} \bigg\| \nabla \frac{\delta E}{\delta m}(\mu,\cdot) \bigg\|^2_2 d\mu + C_{\alpha}.
\end{equation}
This inverse Poincaré-type condition ensures that drift forces always dominate the fluctuations of the stochastic Laplacian.
\end{itemize}
As for the second term, interpreted as the Riemannian inner product of the Wasserstein gradients $\langle \nabla_\mathcal{W}H, \nabla_\mathcal{W} E \rangle_{\dot{H}^{-1}}$, it governs the dynamics of the macro-micro Lagrange multipliers ($\lambda_t$ and $\lambda^n_t$) and proves crucial for closing the Gronwall estimates. Within the nonlinear constrained Fokker-Planck equation \cref{gradflowPDE} and the associated particle system, the numerator of the Lagrange multiplier is driven by the alignment of the forces:
\begin{equation}
\mu\longmapsto\int_{\mathbb{R}^d} \nabla \frac{\delta H}{\delta m}(\mu,\cdot)\cdot\nabla \frac{\delta E}{\delta m}(\mu,\cdot) d\mu.
\end{equation}
To ensure the convergence of the propagation of chaos at the optimal rate of $O(\frac{1}{\sqrt{n}})$, the cross-variations of this functional with respect to the measure $\mu$ must be controlled in a dual manner. To safeguard the system against instabilities arising from poor orthogonality or singular alignment, we introduce the following geometric criteria:
\begin{itemize}
    \item \textbf{Joint lipschitz continuity of force fields:}
The two force fields with flat derivatives must satisfy a joint Lipschitz regularity condition with respect to the Kantorovich-Wasserstein metric $\mathcal{W}_2$ in $L^2(\mu)$:
\begin{equation}
\left| \int_{\mathbb{R}^d} \nabla \frac{\delta H}{\delta m}(\mu,\cdot)\cdot\nabla \frac{\delta E}{\delta m}(\mu,\cdot) d\mu - \int_{\mathbb{R}^d} \nabla \frac{\delta H}{\delta m}(\nu,\cdot)\cdot\nabla \frac{\delta E}{\delta m}(\nu,\cdot) d\nu \right| \le L_{\nabla}\mathcal{W}_2(\mu, \nu) + C_{\nabla} \sqrt{\int_{\mathbb{R}^d} \bigg\|\nabla \frac{\delta H}{\delta m}(\mu,\cdot) - \nabla \frac{\delta H}{\delta m}(\nu,\cdot)\bigg\|^2_2 d\mu}.
\end{equation}
This condition prevents the scalar product from undergoing discontinuous macroscopic variations in response to small microscopic perturbations of the probability distribution.
\item \textbf{Angular coercivity hypothesis (Cauchy-Schwarz lower bound):}
To prevent the numerator from diverging relative to the denominator of the Lagrange multiplier, the growth of the inner product must be dominated in a sub-linear fashion by the squared norm of the constraint, or balanced by the free Fisher information:
\begin{equation}
\left| \int_{\mathbb{R}^d} \nabla \frac{\delta H}{\delta m}(\mu,\cdot)\cdot\nabla \frac{\delta E}{\delta m}(\mu,\cdot)d\mu \right| \le \beta \int_{\mathbb{R}^d} \bigg\|\nabla \frac{\delta E}{\delta m}(\mu,\cdot)\bigg\|^2_2 d\mu + \gamma \mathbf{I}[\mu|dx]
\end{equation}
where $\beta \ge 0$, $\gamma \in [0, \frac{1}{2}[$, and $\mathbf{I}[\cdot|dx]$ is the standard Fisher information.
\end{itemize}

\begin{theorem}[\texttt{Bounds on Lagrange multipliers}]\label{ExempleAppThmBoundLambda}
If $\nabla\frac{\delta H}{\delta m}$ is Lipschitz in its two arguments:
\begin{equation}
    \|\mathcal{D}_mH(\mu,x)-\mathcal{D}_mH(\nu,y)\|_2\leq M^{H}\bigg(\mathcal{W}_2(\mu,\nu)+\|x-y\|_2\bigg)
\end{equation}
or $\mathcal{D}^2_mH$ exists and is bounded:
\begin{equation}
    \|\mathcal{D}^2_mH(\mu,x,y)\|_{\mathbf{op}}\leq M^H_2,
\end{equation}
then we have \cref{TheoremBoundLambda} and $|\lambda^n_t|\leq C_{\lambda}(t)$ almost surely.
\end{theorem}

\begin{proof}
    We have
    \begin{align}
        |\lambda_\mu-\lambda_\nu|&\leq\frac{1}{c_0}\bigg(\sigma^2L_\Delta\mathcal{W}_2(\mu,\nu)+\bigg|\int_{\mathbb{R}^d} \nabla \frac{\delta H}{\delta m}(\mu,\cdot)\cdot\nabla \frac{\delta E}{\delta m}(\mu,\cdot) d\mu-\int_{\mathbb{R}^d} \nabla \frac{\delta H}{\delta m}(\nu,\cdot)\cdot\nabla \frac{\delta E}{\delta m}(\nu,\cdot) d\nu\bigg|\bigg)\\
        &\leq C_{\Delta\nabla}\mathcal{W}_2(\mu,\nu)\nonumber
    \end{align}
    with (Lipschitz first intrinsic derivative)
    \begin{equation}
        C_{\Delta\nabla}:=\frac{1}{c_0}(\sigma^2L_\Delta+L_{\nabla}+C_\nabla M^H)
    \end{equation}
    or (bounded second intrinsic derivative)
    \begin{equation}
        C_{\Delta\nabla}:=\frac{1}{c_0}(\sigma^2L_\Delta+L_{\nabla}+C_\nabla M^H_2).
    \end{equation}
    Furthermore, we have
    \begin{align}
        |\lambda_\mu|\leq\alpha+\frac{C_\alpha}{c_0}+\beta+\frac{\gamma}{c_0}\mathbf{I}[\mu|dx].
    \end{align}
    
\end{proof}

\begin{theorem}[\texttt{Flat derivatives of Lagrange multiplier}]\label{FlatDerivativeLambda}
Assume that $\mathcal{E}$ and $E$ admit first and second flat derivatives with sufficient integrability and regularity conditions, and $\nabla\frac{\delta}{\delta m}=\frac{\delta}{\delta m}\nabla$. In this case, $\mu\longmapsto\lambda_\mu$ admits first flat derivatives:
\begin{align}
    \frac{\delta\lambda}{\delta m}(\mu,\cdot)&=-\frac{1}{\| \nabla \frac{\delta E}{\delta m}(\mu,\cdot)\|^2_{L^2(\mu)}}\bigg(\lambda_\mu\| \nabla \frac{\delta E}{\delta m}(\mu,\cdot)\|^2_2+2\lambda_\mu\int\nabla \frac{\delta E}{\delta m}(\mu,x)\cdot\nabla_x\frac{\delta^2 E}{\delta m^2}(\mu,x,\cdot)\mu(dx)\\
    &+\nabla \frac{\delta \mathcal{E}}{\delta m}(\mu,\cdot)\cdot\nabla \frac{\delta E}{\delta m}(\mu,\cdot)+\int\nabla \frac{\delta \mathcal{E}}{\delta m}(\mu,x)\cdot\nabla_x\frac{\delta^2 E}{\delta m^2}(\mu,x,\cdot)\mu(dx)+\int\nabla \frac{\delta E}{\delta m}(\mu,x)\cdot\nabla_x\frac{\delta^2 \mathcal{E}}{\delta m^2}(\mu,x,\cdot)\mu(dx)\bigg).\nonumber
\end{align}
In particular, if $\nabla\frac{\delta\lambda}{\delta m}$ is bounded, we have
\begin{equation}
    |\lambda_{\mu}-\lambda_{\nu}|\leq\|\nabla\frac{\delta\lambda}{\delta m}\|_{\infty}\mathcal{W}_1(\mu,\nu).
\end{equation}
\end{theorem}
\begin{proof}
    For all $t\in[0,1]$ and $\nu_t:=t\nu_1+(1-t)\nu_0$, we have
    \begin{align}
        \frac{d}{dt}\lambda_{\nu_t}&=-\lambda_{\nu_t}\frac{\frac{d}{dt}\| \nabla \frac{\delta E}{\delta m}(\nu_t,\cdot)\|^2_{L^2(\nu_t)}}{\| \nabla \frac{\delta E}{\delta m}(\nu_t,\cdot)\|^2_{L^2(\nu_t)}}\\
        &-\frac{1}{\| \nabla \frac{\delta E}{\delta m}(\nu_t,\cdot)\|^2_{L^2(\nu_t)}}\frac{d}{dt}\int_{\mathbb{R}^d} \nabla \frac{\delta \mathcal{E}}{\delta m}(\nu_t,\cdot)\cdot\nabla \frac{\delta E}{\delta m}(\nu_t,\cdot) d\nu_t;\nonumber\\
        \frac{d}{dt}\| \nabla \frac{\delta E}{\delta m}(\nu_t,\cdot)\|^2_{L^2(\nu_t)}&=\frac{d}{dt}\bigg\langle\nu_t,\| \nabla \frac{\delta E}{\delta m}(\nu_t,\cdot)\|^2_{2}\bigg\rangle\\
        &=\bigg\langle\nu_1-\nu_0,\| \nabla \frac{\delta E}{\delta m}(\nu_t,\cdot)\|^2_{2}\bigg\rangle+2\bigg\langle\nu_t,\nabla \frac{\delta E}{\delta m}(\nu_t,\cdot)\cdot\frac{d}{dt}\nabla \frac{\delta E}{\delta m}(\nu_t,\cdot)\bigg\rangle;\nonumber\\
        \frac{d}{dt}\int_{\mathbb{R}^d} \nabla \frac{\delta \mathcal{E}}{\delta m}(\nu_t,\cdot)\cdot\nabla \frac{\delta E}{\delta m}(\nu_t,\cdot) d\nu_t&=\frac{d}{dt}\bigg\langle\nu_t,\nabla \frac{\delta \mathcal{E}}{\delta m}(\nu_t,\cdot)\cdot\nabla \frac{\delta E}{\delta m}(\nu_t,\cdot)\bigg\rangle\\
        &=\bigg\langle\nu_t,\nabla \frac{\delta \mathcal{E}}{\delta m}(\nu_t,\cdot)\cdot\frac{d}{dt}\nabla \frac{\delta E}{\delta m}(\nu_t,\cdot)+\nabla \frac{\delta E}{\delta m}(\nu_t,\cdot)\cdot\frac{d}{dt}\nabla \frac{\delta \mathcal{E}}{\delta m}(\nu_t,\cdot)\bigg\rangle\nonumber\\
        &+\bigg\langle\nu_1-\nu_0,\nabla \frac{\delta \mathcal{E}}{\delta m}(\nu_t,\cdot)\cdot\nabla \frac{\delta E}{\delta m}(\nu_t,\cdot)\bigg\rangle;\nonumber\\
        \frac{d}{dt}\nabla \frac{\delta E}{\delta m}(\nu_t,x)&=\bigg\langle\nu_1-\nu_0,\nabla_x\frac{\delta^2 E}{\delta m^2}(\nu_t,x,\cdot)\bigg\rangle;\\
        \frac{d}{dt}\nabla \frac{\delta\mathcal{E}}{\delta m}(\nu_t,x)&=\bigg\langle\nu_1-\nu_0,\nabla_x\frac{\delta^2 \mathcal{E}}{\delta m^2}(\nu_t,x,\cdot)\bigg\rangle.
    \end{align}
    As for the inequality, it is obtained using the Kantorovitch-Rubinstein duality relation:
    \begin{align}
        |\lambda_{\nu_1}-\lambda_{\nu_0}|&=\bigg|\int_0^1\frac{d}{dt}\lambda_{\nu_t}dt\bigg|\\
        &=\bigg|\int_0^1\int\frac{\delta\lambda}{\delta m}(\nu_t,x)d(\nu_1-\nu_0)(dx)dt\bigg|\nonumber\\
        &\leq\|\nabla\frac{\delta\lambda}{\delta m}\|_{\infty}\mathcal{W}_1(\nu_1,\nu_0).\nonumber
    \end{align}
\end{proof}

\section{Constrained mean-field entropy}\label{SectionCMFEntropy}
Let us give the explicit form of the Lagrange multiplier.
\begin{theorem}[\texttt{Explicit form of the Lagrange multiplier}]\label{Thmexplicitformlambda}
    For all $\mu\in\mathcal{C}_M$, we have
    \begin{equation}
        \lambda_\mu=\frac{1}{\| \nabla \frac{\delta E}{\delta m}(\mu,\cdot)\|^2_{L^2(\mu)}}\int\mathcal{G}_\mu\frac{\delta E}{\delta m}(\mu,\cdot)d\mu
    \end{equation}
    with
    \begin{equation}
        \mathcal{G}_\mu:=\sigma^2\Delta-\nabla\frac{\delta H}{\delta m}(\mu,\cdot)\cdot\nabla.
    \end{equation}
\end{theorem}
\begin{proof}
The Lagrange multiplier $\lambda_\mu\in\mathbb{R}$ is defined by the orthogonality property induced by the constraint on the Wasserstein space:
\begin{equation}
    \nabla_{\mathcal{W}}\mathcal{E}(\mu,\cdot)+\lambda_\mu\nabla_{\mathcal{W}}E(\mu,\cdot)=0.
\end{equation}
Since the functionals admit intrinsic derivatives, the latter are their Wasserstein gradients:
\begin{equation}
    \nabla\frac{\delta \mathcal{E}}{\delta m}(\mu,\cdot)+\lambda_\mu\nabla\frac{\delta E}{\delta m}(\mu,\cdot)=0.
\end{equation}
By taking the dot product and using the divergence formula, we have
\begin{align}
\lambda_\mu&= -\frac{1}{\| \nabla \frac{\delta E}{\delta m}(\mu,\cdot)\|^2_{L^2(\mu)}}\int_{\mathbb{R}^d} \nabla \frac{\delta \mathcal{E}}{\delta m}(\mu,\cdot)\cdot\nabla \frac{\delta E}{\delta m}(\mu,\cdot) d\mu\\
&=\frac{1}{\| \nabla \frac{\delta E}{\delta m}(\mu,\cdot)\|^2_{L^2(\mu)}}\bigg(\sigma^2\int_{\mathbb{R}^d}\Delta\frac{\delta E}{\delta m}(\mu,\cdot)d\mu-\int_{\mathbb{R}^d} \nabla \frac{\delta H}{\delta m}(\mu,\cdot)\cdot\nabla \frac{\delta E}{\delta m}(\mu,\cdot)d\mu\bigg).\nonumber
\end{align}
\end{proof}

\subsection*{\underline{Macroscopic Lyapunov functional}}

\begin{proposition}[\texttt{Free Lyapunov functional}]\label{PropoFreeEnergyLyapunov}
    For the constrained flow $(\mu_t)_{t\geq0}$ given by \cref{gradflowPDE}, free energy $\mathcal{E}:=H+\sigma^2\mathbf{H}[\cdot|dx]$ is a Lyapunov functional.
\end{proposition}
\begin{proof}
    By definition of flat derivatives, \cref{Thmexplicitformlambda} and Cauchy-Schwarz inequality, we have
    \begin{align}
    \frac{d}{dt}\mathcal{E}(\mu_t)&=\int\frac{\delta\mathcal{E}}{\delta m}(\mu_t,\cdot)d\partial_t\mu_t\\
    &=-\int\nabla\frac{\delta\mathcal{E}}{\delta m}(\mu_t,\cdot)\cdot\nabla\bigg(\frac{\delta\mathcal{E}}{\delta m}(\mu_t,\cdot)+\lambda_t\frac{\delta E}{\delta m}(\mu_t,\cdot)\bigg)d\mu_t\nonumber\\
    &=-\int\bigg\|\nabla\frac{\delta\mathcal{E}}{\delta m}(\mu_t,\cdot)\bigg\|^2_2d\mu_t+\frac{1}{\| \nabla \frac{\delta E}{\delta m}(\mu_t,\cdot)\|^2_{L^2(\mu_t)}}\bigg(\int\nabla\frac{\delta\mathcal{E}}{\delta m}(\mu_t,\cdot)\cdot\nabla\frac{\delta E}{\delta m}(\mu_t,\cdot)d\mu_t\bigg)^2\nonumber\\
    &\leq-\int\bigg\|\nabla\frac{\delta\mathcal{E}}{\delta m}(\mu_t,\cdot)\bigg\|^2_2d\mu_t+\int\bigg\|\nabla\frac{\delta\mathcal{E}}{\delta m}(\mu_t,\cdot)\bigg\|^2_2d\mu_t=0.\nonumber
\end{align}
\end{proof}
To study the evolution of the constrained mean-field entropy along the constrained flow, we define a constrained mean-field Fisher information by
\begin{align}
    \mathcal{I}(\mu):=\int \bigg\|\nabla\bigg(\frac{\delta\mathcal{E}}{\delta m}(\mu,\cdot)+\lambda_\mu\frac{\delta E}{\delta m}(\mu,\cdot)\bigg)\bigg\|^2_2d\mu.
\end{align}
\begin{proposition}[\texttt{Constrained entropy production}]\label{PropoEntropyProd}
    We have
    \begin{equation}
        \frac{d}{dt}\mathcal{E}(\mu_t)=-\mathcal{I}(\mu_t).
    \end{equation}
\end{proposition}
\begin{proof}
    By bilinearity and \cref{Thmexplicitformlambda}, we have
    \begin{align}
        \mathcal{I}(\mu_t)&=\int \bigg\|\nabla\bigg(\frac{\delta\mathcal{E}}{\delta m}(\mu_t,\cdot)+\lambda_t\frac{\delta E}{\delta m}(\mu_t,\cdot)\bigg)\bigg\|^2_2d\mu_t\\
        &=\int \bigg\|\nabla\frac{\delta\mathcal{E}}{\delta m}(\mu_t,\cdot)\bigg\|^2_2d\mu_t+2\lambda_t\int\nabla\frac{\delta\mathcal{E}}{\delta m}(\mu_t,\cdot)\cdot\nabla\frac{\delta E}{\delta m}(\mu_t,\cdot)d\mu_t+\lambda_t^2\int \bigg\|\nabla\frac{\delta E}{\delta m}(\mu_t,\cdot)\bigg\|^2_2d\mu_t\nonumber\\
        &=-\mathcal{E}(\mu_t).\nonumber
    \end{align}
\end{proof}

\subsection*{\underline{Microscopic Lyapunov functional}}
Under \cref{ThmExplicitMicroscopicIM}, we have
\begin{equation}
    \mu^n_\infty(d\mathbf{x}):=\frac{1}{Z_n}e^{-\frac{1}{\sigma^2}\bigg(H_n(\mathbf{x})+\psi_n(E_n(\mathbf{x}))\bigg)}d\mathbf{x}
\end{equation}
with a scalar function $\psi_n\in\mathcal{C}^1$.
\begin{proposition}[\texttt{Free microscopic Lyapunov functional}]\label{propoFreeMicroscopicEnergy}
    For the microscopic constrained flow $(\mu^n_t)_{t\geq0}$ given by \cref{microscopicflowPDE}, the relative entropy $\sigma^2\mathbf{H}[\cdot|\mu^n_\infty]$ is a Lyapunov functional:
    \begin{equation}
        \sigma^2\frac{d}{dt}\mathbf{H}[\mu^n_t|\mu^n_\infty]=-\int\bigg\|\nabla H_n+\xi_n\nabla E_n+\sigma^2\nabla\ln\frac{d\mu^n_t}{d\mathbf{x}}\bigg\|^2_2d\mu^n_t.
    \end{equation}
\end{proposition}
\begin{proof}
    We have
    \begin{align}
        \sigma^2\frac{d}{dt}\mathbf{H}[\mu^n_t|\mu^n_\infty]&=\int\frac{\delta\mathbf{H}[\cdot|\mu^n_\infty]}{\delta m}(\mu^n_t,\cdot)d\partial_t\mu^n_t\\
        &=-\sigma^2\int\nabla\frac{\delta\mathbf{H}[\cdot|\mu^n_\infty]}{\delta m}(\mu^n_t,\cdot)\cdot\bigg(\nabla H_n+\xi_n\nabla E_n+\sigma^2\nabla\ln\frac{d\mu^n_t}{d\mathbf{x}}\bigg)d\mu^n_t\nonumber\\
        &=-\int\bigg\|\nabla H_n+\xi_n\nabla E_n+\sigma^2\nabla\ln\frac{d\mu^n_t}{d\mathbf{x}}\bigg\|^2_2d\mu^n_t\nonumber.
    \end{align}
\end{proof}
We define a constrained microscopic Fisher information by
\begin{equation}
    \mathbf{I}[\mu|\mu^n_\infty]:=\int\bigg\|\nabla H_n+\xi_n\nabla E_n+\sigma^2\nabla\ln\frac{d\mu}{d\mathbf{x}}\bigg\|^2_2d\mu.
\end{equation}
Note that
\begin{align}
    \sigma^2\mathbf{H}[\mu|\mu^n_\infty]&=\sigma^2\int\ln\frac{d\mu}{d\mu^n_\infty}d\mu\\
    &=\sigma^2\mathbf{H}[\mu|d\mathbf{x}]+\int\bigg( H_n(\mathbf{x})+\psi(E_n(\mathbf{x}))\bigg)\mu(d\mathbf{x})+\sigma^2\ln Z_n.\nonumber
\end{align}

\section{Examples and applications}\label{SectionExamples}
In this section, we present concrete and original examples and applications. We recall that
\begin{align}
    \lambda_\mu=\frac{1}{\| \nabla \frac{\delta E}{\delta m}(\mu,\cdot)\|^2_{L^2(\mu)}}\bigg(\sigma^2\int_{\mathbb{R}^d}\Delta\frac{\delta E}{\delta m}(\mu,\cdot)d\mu-\int_{\mathbb{R}^d} \nabla \frac{\delta H}{\delta m}(\mu,\cdot)\cdot\nabla \frac{\delta E}{\delta m}(\mu,\cdot)d\mu\bigg).
\end{align}

\subsection*{\underline{Polynomial constraint}}
We consider $N$ potentials $W^{(k)}:(\mathbb{R}^d)^k\longrightarrow\mathbb{R}\in\mathcal{C}^2$ with quadratic growth such that
\begin{equation}
    E(\mu):=\sum_{k=1}^{N}\int W^{(k)}d\mu^{\otimes k}.
\end{equation}

\begin{example}[\texttt{Linear case}]\label{Example1}
In this case, we have
\begin{equation}
    \lambda_\mu=\frac{1}{\|\nabla W^{(1)}\|^2_{L^2(\mu)}}\bigg(\sigma^2\int\Delta W^{(1)} d\mu-\int\nabla\frac{\delta H}{\delta m}(\mu,\cdot)\cdot\nabla W^{(1)}d\mu\bigg).
\end{equation}
In particular, if the potential is harmonic ($\Delta W^{(1)}=0$), we have
\begin{equation}
    \lambda_\mu=-\frac{1}{\|\nabla W^{(1)}\|^2_{L^2(\mu)}}\int\nabla\frac{\delta H}{\delta m}(\mu,\cdot)\cdot\nabla W^{(1)}d\mu\bigg).
\end{equation}
By Cauchy-Schwarz inequality, we have
\begin{equation}
    |\lambda_\mu|\leq\frac{\|\nabla\frac{\delta H}{\delta m}(\mu,\cdot)\|_{L^2(\mu)}}{\|\nabla W^{(1)}\|_{L^2(\mu)}}\leq\frac{\|\nabla\frac{\delta H}{\delta m}\|_{\infty}}{C_W}.
\end{equation}
If the potential is a linear form ($x\mapsto\langle v,x\rangle$), the constraint characterizes the conservation of a combination of marginal moments, and we have
\begin{equation}
    |\lambda_\mu|\leq\frac{\|\nabla\frac{\delta H}{\delta m}\|_{\infty}}{\|v\|}.
\end{equation}
Note that this case greatly generalizes the constraint $\int_0^{+\infty} x\mu(dx)=M$ of \cite{10.1098/rsos.260620}. Moreover, the submanifold is totally geodesic.
\end{example}

\begin{example}[\texttt{Quadratic constraint}]\label{Example2}
    In this case, we consider
    \begin{equation}
        E(\mu):=\int W^{(2)}d\mu^{\otimes 2}.
    \end{equation}
    This type of constraint makes it possible, for example, to fix the variance or a linear combination of second-order moments.
\end{example}

\subsection*{\underline{Internal constraint}}
Another family of constraints is given by
\begin{equation}
    E(\mu)=\Psi\bigg(\int Wd\mu\bigg)
\end{equation}
with $\Psi$ a scalar function.

\section{Proof of main theorems}\label{ProofmainTheorems}
\begin{proof}[Proof of \cref{Theorem1}]
The argument follows the minimizing‑movement approach of \cite{JordanKindrlehererOtto} and the theory of curves
of maximal slope in $(\mathcal{P}_2,\mathcal{W}_2)$ as developed in \cite{AmbrosioGigliSavare}.\newline

\texttt{Discrete scheme.} Fix $\tau>0$ and define recursively
\begin{equation}\label{MethodGradientProximal}
    \nu^\tau_{n+1}\in\textbf{argmin}_{\nu\in\mathcal{C}_{M}}\bigg\{\mathcal{E}(\nu)+\frac{1}{2\tau}\mathcal{W}^2_{2}(\nu,\nu^\tau_{n}) \bigg\}\quad\textnormal{with }\nu^\tau_0=\mu_0.
\end{equation}
Since $\mathcal{C}_M\subset\mathcal{P}_2(\mathbb{R}^d)$, each iterate $\nu^\tau_n$ is a probability measure (hence non‑negative and of unit mass), and the energy constraint $E(\nu^\tau_n)=M$ holds by construction.\newline

\texttt{Existence of minimizers and compactness.} By coercivity and narrow lower semi‑continuity of $\mathcal{E}$ (assumption \ref{Assumption1}) and closedness of $\mathcal{C}_M$, the direct method yields existence of minimizers of \cref{MethodGradientProximal}. Moreover, each step satisfies the discrete energy inequality,
\begin{equation}
    \mathcal{E}(\nu^\tau_{n+1})+\frac{1}{2\tau}\mathcal{W}^2_2(\nu^\tau_{n+1},\nu^\tau_n)\leq\mathcal{E}(\nu^\tau_n),
\end{equation}
which implies uniform second‑moment bounds. Defining the piecewise‑constant interpolation 
\begin{equation}
    \nu^\tau_t:=\sum_{n=0}^{+\infty}\nu^\tau_n\mathbb{I}_{(n\tau,(n+1)\tau]}(t),
\end{equation}
tightness in $\mathcal{P}_2$ yields (up to subsequences) narrow convergence $\nu^{\tau_n}\overset{n\to+\infty}{\longrightarrow}\nu$ in $\mathcal{C}([0,T],\mathcal{P}_2(\mathbb{R}^d))$.\newline

\texttt{Limit equation and positivity.} By \cite[Theorem 11.1.3]{AmbrosioGigliSavare}, the limit curve $(\nu_t)_{t\geq0}$ is a curve of maximal slope for $\mathcal{E}$ restricted to $\mathcal{C}_M$ and therefore solves equation \cref{gradflowPDE} in the weak sense with some measurable multiplier $\lambda_t$ enforcing $E(\nu_t)=M$. Since each $\nu^{\tau_n}_t$ is a probability measure, the limit $\nu_t$ is also non‑negative with unit mass.\newline

\texttt{Invariance of mass and energy.} For all test function $\varphi$, we have
\begin{align}
    \frac{d}{dt}\int\varphi d\mu_t=\int\varphi d\partial_t\mu_t=-\int\nabla\varphi\cdot\nabla\bigg(\frac{\delta\mathcal{E}}{\delta m}(\mu_t,\cdot)+\lambda_t\frac{\delta E}{\delta m}(\mu_t,\cdot)\bigg)d\mu_t.
\end{align}
Choosing $\varphi\equiv 1$ yields $\int d\mu_t=\int d\mu_0$. On the other hand, we have
\begin{equation}\label{EqLyapunov}
 \frac{d}{dt}E(\mu_t)=\int\frac{\delta E}{\delta m}(\mu_t,\cdot)d\partial_t\mu_t=-\int\nabla\frac{\delta E}{\delta m}(\mu_t,\cdot)\cdot\nabla\bigg(\frac{\delta\mathcal{E}}{\delta m}(\mu_t,\cdot)+\lambda_t\frac{\delta E}{\delta m}(\mu_t,\cdot)\bigg)d\mu_t.
\end{equation}
It follows that $E(\mu_t)=E(\mu_0)$.
\end{proof}

\begin{proof}[Proof of \cref{Theorem2}]
    For the discrete sequence $(\nu^\tau_{n})_{n\geq0}$ one has the standard JKO estimate,
    \begin{equation}
        \mathcal{E}(\nu^\tau_{n+1})-\mathcal{E}(\nu^\tau_n)\leq-\frac{1}{2\tau}\mathcal{W}^2_2(\nu^\tau_{n+1},\nu^\tau_n).
    \end{equation}
    Define the discrete metric derivative
    \begin{equation}
        |\partial_t\nu^\tau_t|_{\mathcal{W}_2}:=\frac{1}{\tau}\sum_{n=0}^{+\infty}\mathcal{W}_2(\nu^\tau_{n+1},\nu^\tau_n)\mathbb{I}_{(n\tau,(n+1)\tau]}.
    \end{equation}
    Summing over $n$ yields a discrete energy dissipation bound. Passing to the limit along $(\tau_n)_{n}$ and invoking the general convergence theory for curves of maximal slope in metric spaces \cite[Theorem 11.1.4]{AmbrosioGigliSavare}  gives \cref{EnergyDissipationInequality}.
\end{proof}

\begin{proof}[Proof of \cref{Theorem3}]
    Under $\lambda$‑geodesic convexity of $\mathcal{E}$ on $\mathcal{C}_{M}$, the constrained flow satisfies the evolution variational inequality in the sense of \cite[Definition 4.0.4]{AmbrosioGigliSavare}:
    \begin{equation}
    \forall\nu\in\mathcal{C}_M,\quad\frac{1}{2}\frac{d}{dt}\mathcal{W}^2_{2}(\mu_t,\nu)+\frac{\lambda}{2}\mathcal{W}^2_{2}(\mu_t,\nu)\leq\mathcal{E}(\nu)-\mathcal{E}(\mu_t).
\end{equation}
Taking $\nu=\mu_\infty$ (the unique minimizer of $\mathcal{E}$ on $\mathcal{C}_M$) and applying Grönwall’s inequality yields the
exponential convergence bounds stated in \cref{Theorem3}. 
\end{proof}

\begin{proof}[Proof of \cref{Theorem4}]
     We have
\begin{align}
    \frac{1}{2}\frac{d}{dt}\|X^{n,i}_t-X^i_t\|^2_2&=\bigg\langle\frac{d}{dt}X^{n,i}_t-\frac{d}{dt}X^i_t,X^{n,i}_t-X^i_t\bigg\rangle\\
    &=\bigg\langle\bigg(\nabla\frac{\delta H}{\delta m}(\mu_t,X^i_t)-\nabla\frac{\delta H}{\delta m}(\mu_{\mathbf{X}_t},X^{n,i}_t)\bigg)+\lambda_t\nabla\frac{\delta E}{\delta m}(\mu_t,X^i_t)-\lambda^n_t\nabla\frac{\delta E}{\delta m}(\mu_{\mathbf{X}_t},X^{n,i}_t),X^{n,i}_t-X^i_t\bigg\rangle;\nonumber\\
    \lambda_t\nabla\frac{\delta E}{\delta m}(\mu_t,X^i_t)-\lambda^n_t\nabla\frac{\delta E}{\delta m}(\mu_{\mathbf{X}_t},X^{n,i}_t)&=(\lambda_t-\lambda^n_t)\nabla\frac{\delta E}{\delta m}(\mu_t,X^i_t)+\lambda^n_t\bigg(\nabla\frac{\delta E}{\delta m}(\mu_t,X^i_t)-\nabla\frac{\delta E}{\delta m}(\mu_{\mathbf{X}_t},X^{n,i}_t)\bigg)
\end{align}
It follows that (Cauchy-Schwarz inequality and \cref{TheoremBoundLambda})
\begin{align}
    \frac{1}{2}\frac{d}{dt}\|X^{n,i}_t-X^i_t\|^2_2&\leq (M^H+|\lambda^n_t|M^E)\bigg(\mathcal{W}_2(\mu_{\mathbf{X}_t},\mu_t)+\|X^{n,i}_t-X^i_t\|_2\bigg)\|X^{n,i}_t-X^i_t\|_2\\
    &+|\lambda^n_t-\lambda_t|M^E_1\|X^{n,i}_t-X^i_t\|_2\nonumber\\
    &\leq M^H\bigg(\mathcal{W}_2(\mu_{\mathbf{X}_t},\mu_t)\|X^{n,i}_t-X^i_t\|_2+\|X^{n,i}_t-X^i_t\|^2_2\bigg)\nonumber\\
    &+M^E|\lambda^n_t|\bigg(\mathcal{W}_2(\mu_{\mathbf{X}_t},\mu_t)\|X^{n,i}_t-X^i_t\|_2+\|X^{n,i}_t-X^i_t\|^2_2\bigg)\nonumber\\
    &+M^E_1|\lambda^n_t-\lambda_t|\cdot\|X^{n,i}_t-X^i_t\|_2\nonumber\\
    &\leq M^H\bigg(\mathcal{W}_2(\mu_{\mathbf{X}_t},\mu_t)\|X^{n,i}_t-X^i_t\|_2+\|X^{n,i}_t-X^i_t\|^2_2\bigg)\nonumber\\
    &+M^EC_\lambda(t)\bigg(\mathcal{W}_2(\mu_{\mathbf{X}_t},\mu_t)\|X^{n,i}_t-X^i_t\|_2+\|X^{n,i}_t-X^i_t\|^2_2\bigg)\nonumber\\
    &+M^E_1C_{\Delta\nabla}\mathcal{W}_2(\mu_{\mathbf{X}_t},\mu_t)\|X^{n,i}_t-X^i_t\|_2\nonumber.
\end{align}
By triangular inequality, we have
\begin{equation}
    \mathcal{W}_2(\mu_{\mathbf{X}_t},\mu_t)\leq\mathcal{W}_2(\mu_{\mathbf{X}_t},\mu_{\overline{\mathbf{X}}_t})+\mathcal{W}_2(\mu_{\overline{\mathbf{X}}_t},\mu_t).
\end{equation}
By Young's inequality, we have
\begin{align}
    \frac{d}{dt}\|X^{n,i}_t-X^i_t\|^2_2&\leq(M^H+C_\lambda(t)M^E)\bigg(\mathcal{W}^2_2(\mu_{\mathbf{X}_t},\mu_{\overline{\mathbf{X}}_t})+\mathcal{W}^2_2(\mu_{\overline{\mathbf{X}}_t},\mu_t)+3\|X^{n,i}_t-X^i_t\|^2_2\bigg)\\
    &+M^E_1C_{\Delta\nabla}\bigg(\mathcal{W}^2_2(\mu_{\mathbf{X}_t},\mu_{\overline{\mathbf{X}}_t})+\mathcal{W}^2_2(\mu_{\overline{\mathbf{X}}_t},\mu_t)+2\|X^{n,i}_t-X^i_t\|^2_2\bigg).\nonumber
\end{align}
By symmetry, we have
\begin{equation}
    \forall i\in\{1,\ldots,n\},\quad\mathbb{E}[\|X^{n,i}_t-X^i_t\|^2_2]=\frac{1}{n}\sum_{j=1}^{n}\mathbb{E}[\|X^{n,j}_t-X^j_t\|^2_2]\geq\mathbb{E}[\mathcal{W}^2_2(\mu_{\mathbf{X}_t},\mu_{\overline{\mathbf{X}}_t})].
\end{equation}
Setting $\gamma_t:=\mathbb{E}[\|X^{n,i}_t-X^i_t\|^2_2]$, we have
\begin{align}
    \frac{d}{dt}\gamma_t\leq \bigg(4M^H+4C_{\lambda}(t)M^E+3C_{\Delta\nabla}M^E_1\bigg)\gamma_t+\bigg(M^H+C_\lambda(t)M^E+C_{\Delta\nabla}M^E_1\bigg)\mathbb{E}[\mathcal{W}^2_2(\mu_{\overline{\mathbf{X}}_t},\mu_t)].
\end{align}
By the result of \cite[Theorem 1]{FournierGuillin}, we have
\begin{equation}
    \mathbb{E}[\mathcal{W}^2_2(\mu_{\overline{\mathbf{X}}_t},\mu_{t})]\leq C(d)M(\mu_t)\delta_d(n).
\end{equation}
We have (Gronwall's lemma and $\gamma_0=0$)
\begin{align}
    \gamma_t\leq C(d)\delta_d(n)\int_{0}^{t}\bigg(M^H+C_\lambda(s)M^E+C_{\Delta\nabla}M^E_1\bigg)M(\mu_s)e^{\int_{s}^{t}\bigg(4M^H+4C_{\lambda}(u)M^E+3C_{\Delta\nabla}M^E_1\bigg)du}ds.
\end{align}
\end{proof}

\begin{proof}[Proof of \cref{Theorem5}]
 We have
\begin{align}
    \frac{1}{2}\frac{d}{dt}\|X^{n,i}_t-X^i_t\|^2_2&=\bigg\langle\frac{d}{dt}X^{n,i}_t-\frac{d}{dt}X^i_t,X^{n,i}_t-X^i_t\bigg\rangle\\
    &=\bigg\langle\bigg(\nabla\frac{\delta H}{\delta m}(\mu_t,X^i_t)-\nabla\frac{\delta H}{\delta m}(\mu_{\mathbf{X}_t},X^{n,i}_t)\bigg)+\lambda_t\nabla\frac{\delta E}{\delta m}(\mu_t,X^i_t)-\lambda^n_t\nabla\frac{\delta E}{\delta m}(\mu_{\mathbf{X}_t},X^{n,i}_t),X^{n,i}_t-X^i_t\bigg\rangle;\nonumber\\
    \nabla\frac{\delta H}{\delta m}(\mu_t,X^i_t)-\nabla\frac{\delta H}{\delta m}(\mu_{\mathbf{X}_t},X^{n,i}_t)&=\bigg(\nabla\frac{\delta H}{\delta m}(\mu_t,X^i_t)-\nabla\frac{\delta H}{\delta m}(\mu_t,X^{n,i}_t)\bigg)+\nabla\frac{\delta H}{\delta m}(\mu_t,X^{n,i}_t)-\nabla\frac{\delta H}{\delta m}(\mu_{\mathbf{X}_t},X^{n,i}_t);\\
    \lambda_t\nabla\frac{\delta E}{\delta m}(\mu_t,X^i_t)-\lambda^n_t\nabla\frac{\delta E}{\delta m}(\mu_{\mathbf{X}_t},X^{n,i}_t)&=(\lambda_t-\lambda^n_t)\nabla\frac{\delta E}{\delta m}(\mu_t,X^i_t)+\lambda^n_t\bigg(\nabla\frac{\delta E}{\delta m}(\mu_t,X^i_t)-\nabla\frac{\delta E}{\delta m}(\mu_{\mathbf{X}_t},X^{n,i}_t)\bigg)
\end{align}
By Cauchy-Schwarz, triangular and Young inequalities, Kantorovitch-Rubinstein duality relation and \cref{TheoremBoundLambda}, it follows that
\begin{align}
    \frac{d}{dt}\|X^{n,i}_t-X^i_t\|^2_2&\leq -2\lambda_H\|X^{n,i}_t-X^i_t\|^2_2\\
    &+2M^H_2\mathcal{W}_2(\mu_t,\mu_{\mathbf{X}_t})\|X^{n,i}_t-X^i_t\|_2\nonumber\\
    &+2|\lambda^n_t|M^E\bigg(\mathcal{W}_2(\mu_{\mathbf{X}_t},\mu_t)\|X^{n,i}_t-X^i_t\|_2+\|X^{n,i}_t-X^i_t\|^2_2\bigg)\nonumber\\
    &+2|\lambda^n_t-\lambda_t|M^E_1\|X^{n,i}_t-X^i_t\|_2\nonumber\\
    &\leq (2C_\lambda(t)M^E-2\lambda_H)\|X^{n,i}_t-X^i_t\|^2_2\nonumber\\
    &+2(M^H_2+C_\lambda(t)M^E+C_{\Delta\nabla}M^E_1)\mathcal{W}_2(\mu_t,\mu_{\mathbf{X}_t})\|X^{n,i}_t-X^i_t\|_2\nonumber\\
    &\leq(2M^H_2+4C_\lambda(t)M^E+2C_{\Delta\nabla}M^E_1-2\lambda_H)\|X^{n,i}_t-X^i_t\|^2_2\nonumber\\
    &+(M^H_2+C_\lambda(t)M^E+C_{\Delta\nabla}M^E_1)\bigg(\mathcal{W}^2_2(\mu_{\mathbf{X}_t},\mu_{\overline{\mathbf{X}}_t})+\mathcal{W}^2_2(\mu_{\overline{\mathbf{X}}_t},\mu_t)\bigg)\nonumber.
\end{align}
By symmetry, we have
\begin{equation}
    \forall i\in\{1,\ldots,n\},\quad\mathbb{E}[\|X^{n,i}_t-X^i_t\|^2_2]=\frac{1}{n}\sum_{j=1}^{n}\mathbb{E}[\|X^{n,j}_t-X^j_t\|^2_2]\geq\mathcal{W}^2_2(\mu_{\mathbf{X}_t},\mu_{\overline{\mathbf{X}}_t}).
\end{equation}
Setting $\gamma_t:=\mathbb{E}[\|X^{n,i}_t-X^i_t\|^2_2]$, we have
\begin{align}
    \frac{d}{dt}\gamma_t\leq\bigg(3M^H_2+5C_\lambda(t)M^E+3C_{\Delta\nabla}M^E_1-2\lambda_H\bigg)\gamma_t+\bigg(M^H_2+C_\lambda(t)M^E+C_{\Delta\nabla}M^E_1\bigg)\mathbb{E}[\mathcal{W}^2_2(\mu_{\overline{\mathbf{X}}_t},\mu_{t})].
\end{align}
By the result of \cite[Theorem 1]{FournierGuillin}, we have
\begin{equation}
    \mathbb{E}[\mathcal{W}^2_2(\mu_{\overline{\mathbf{X}}_t},\mu_{t})]\leq C(d)M(\mu_t)\delta_d(n).
\end{equation}
We have (Gronwall's lemma and $\gamma_0=0$)
\begin{align}
    \gamma_t&\leq C(d)\delta_d(n)\int_{0}^{t}\bigg(M^H_2+C_\lambda(s)M^E+C_{\Delta\nabla}M^E_1\bigg)M(\mu_s)e^{\int_{s}^{t}\bigg(3M^H_2+5C_\lambda(u)M^E+3C_{\Delta\nabla}M^E_1-2\lambda_H\bigg)du}ds.
\end{align}   
\end{proof}

\begin{proof}[Proof of \cref{Theorem6}]
By \cref{TheoremBoundLambda}, we have
\begin{align}
    \frac{d}{dt}\|X_t-Y_t\|^2_2&=-2\bigg\langle\nabla\frac{\delta H}{\delta m}(\mu_t,X_t)-\nabla\frac{\delta H}{\delta m}(\nu_t,Y_t),X_t-Y_t\bigg\rangle-2\bigg\langle\lambda_t\nabla\frac{\delta E}{\delta m}(\mu_t,X_t)-\beta_t\nabla\frac{\delta E}{\delta m}(\nu_t,Y_t),X_t-Y_t\bigg\rangle\\
    &\leq2(M^H+|\beta_t|M^E)\bigg(\mathcal{W}_2(\mu_t,\nu_t)+\|X_t-Y_t\|_2\bigg)\|X_t-Y_t\|_2+2|\lambda_t-\beta_t|M^E_1\|X_t-Y_t\|_2\nonumber\\
    &\leq2(M^H+|\beta_t|M^E)\bigg(\sqrt{\mathbb{E}[\|X_t-Y_t\|^2_2]}+\|X_t-Y_t\|_2\bigg)\|X_t-Y_t\|_2+2C_{\Delta\nabla}M^E_1\sqrt{\mathbb{E}[\|X_t-Y_t\|^2_2]}\|X_t-Y_t\|_2\nonumber.
\end{align}
Setting $\gamma_t:=\mathbb{E}[\|X_t-Y_t\|^2_2]$, we have
\begin{equation}
    \frac{d}{dt}\gamma_t\leq(4M^H+4|\beta_t|M^E+2C_{\Delta\nabla}M^E_1)\gamma_t.
\end{equation}
By Gronwall's lemma, we have 
\begin{equation}
    \gamma_t\leq \gamma_0e^{\int_0^t(4M^H+4|\beta_s|M^E+2C_{\Delta\nabla}M^E_1)ds}.
\end{equation}
We conclude with the existence of an optimal coupling.
\end{proof}

\begin{proof}[Proof of \cref{Theorem7}]
    By \cref{TheoremBoundLambda}, we have
\begin{align}
    \frac{d}{dt}\|X_t-Y_t\|^2_2&=-2\bigg\langle\nabla\frac{\delta H}{\delta m}(\mu_t,X_t)-\nabla\frac{\delta H}{\delta m}(\nu_t,Y_t),X_t-Y_t\bigg\rangle-2\bigg\langle\lambda_t\nabla\frac{\delta E}{\delta m}(\mu_t,X_t)-\beta_t\nabla\frac{\delta E}{\delta m}(\nu_t,Y_t),X_t-Y_t\bigg\rangle\\
    &\leq-2\lambda_H\|X_t-Y_t\|^2_2+2M^H_2\mathcal{W}_1(\mu_t,\nu_t)\|X_t-Y_t\|_2+2|\beta_t|M^E\bigg(\mathcal{W}_2(\mu_t,\nu_t)+\|X_t-Y_t\|_2\bigg)\|X_t-Y_t\|_2\nonumber\\
    &+2|\lambda_t-\beta_t|M^E_1\|X_t-Y_t\|_2\nonumber\\
    &\leq-2\lambda_H\|X_t-Y_t\|^2_2+2M^H_2\mathbb{E}[\|X_t-Y_t\|_2]\|X_t-Y_t\|_2+2|\beta_t|M^E\bigg(\sqrt{\mathbb{E}[\|X_t-Y_t\|^2_2]}+\|X_t-Y_t\|_2\bigg)\|X_t-Y_t\|_2\nonumber\\
    &+2C_{\Delta\nabla}M^E_1\sqrt{\mathbb{E}[\|X_t-Y_t\|^2_2]}\|X_t-Y_t\|_2\nonumber.
\end{align}
Setting $\gamma_t:=\mathbb{E}[\|X_t-Y_t\|^2_2]$, we have
\begin{equation}
    \frac{d}{dt}\gamma_t\leq (2M^H_2+4|\beta_t|M^E+2C_{\Delta\nabla}M^E_1-2\lambda_H)\gamma_t.
\end{equation}
By Gronwall's lemma, we have 
\begin{equation}
    \gamma_t\leq \gamma_0e^{\int_0^t\bigg(2M^H_2+4|\beta_s|M^E+2C_{\Delta\nabla}M^E_1-2\lambda_H\bigg)ds}.
\end{equation}
We conclude with the existence of an optimal coupling.
\end{proof}

\begin{proof}[Proof of \cref{Theorem8}]
     Let $\nu_{0},\nu_{1}\in {\cal C}_M$ and set $\nu_{t}=(1-t)\nu_{0}+t\nu_{1}$, $t\in[0,1]$. Let $f:\mathbb{R}^d\rightarrow\mathbb{R}$ be a $1$-Lipschitz smooth function. From the very definition of $\Phi$,
\begin{equation}\label{eq:phinut3}
    \langle\Phi(\nu_{t}),f\rangle=\frac{1}{Z_{\nu_{t}}}\int_{\mathbb{R}^{d}}f(x)e^{-\frac{1}{\sigma^2}\bigg(\frac{\delta H}{\delta m}(\nu_{t},x)+\lambda_{\nu_t}\frac{\delta E}{\delta m}(\nu_t,x)\bigg)}dx,
\end{equation}
so that setting 
\begin{equation}
    g_t(x):=-\frac{1}{\sigma^2}\frac{d}{dt}\bigg(\frac{\delta H}{\delta m}(\nu_{t},x)+\lambda_{\nu_t}\frac{\delta E}{\delta m}(\nu_t,x)\bigg),
\end{equation}
we get
\begin{align}
\frac{d}{dt}\langle\Phi(\nu_{t}),f\rangle&=-\frac{\partial_t Z_{\nu_t}}{Z_{\nu_t}}\langle\Phi(\nu_{t}),f\rangle+\langle\Phi(\nu_{t}),f g_t \rangle\\
&=\langle\Phi(\nu_{t}),f g_t \rangle-\langle\Phi(\nu_{t}),f\rangle\langle\Phi(\nu_{t}),g_t\rangle\nonumber\\
&={\rm Cov}_{\Phi(\nu_{t})}(f,g_t).\nonumber
\end{align}
As the functionals $H$ and $E$ admit first and second-order flat derivatives, we have
\begin{equation}
    g_t(x):=-\frac{1}{\sigma^2}\bigg(\int_{\mathbb{R}^{d}}\frac{\delta^2 H}{\delta m^2}(\nu_{t},x,y)(\nu_{1}-\nu_{0})(dy)+\lambda_{\nu_t}\int_{\mathbb{R}^{d}}\frac{\delta^2 E}{\delta m^2}(\nu_{t},x,y)(\nu_{1}-\nu_{0})(dy)+\frac{\delta E}{\delta m}(\nu_t,x)\frac{d}{dt}\lambda_{\nu_t}\bigg).
\end{equation}
For a probability $\mu$, let ${\cal L}_{\mu}$ be the operator defined on ${\cal C}^2$-functions by 
\begin{equation}\label{eq:operator3}
 {\cal L}_\mu f=\sigma^2\Delta f-\bigg(\mathcal{D}_{m}H(\mu,\cdot)+\lambda_\mu\mathcal{D}_{m}E(\mu,\cdot)\bigg)\cdot\nabla f.
 \end{equation}
 Denoting by $\varphi_t$ be the solution of the Poisson equation $f-\Phi(\nu_{t})(f)={\cal L}_{\nu_t} \varphi_t$ and using that ${\cal L}_{\nu_t}$ is self-adjoint in $L^2(\Phi(\nu_{t}))$, we get
$${\rm Cov}_{\Phi(\nu_{t})}(f,g_t)=\langle \varphi_t, {\cal L}_{\nu_t} g_t\rangle_{L^2(\Phi(\nu_{t}))}=-\langle \nabla \varphi_t, \nabla g_t\rangle_{L^2(\Phi(\nu_{t}))}.$$
Note that for the second equality, we used the fact that for some ${\cal C}^2$-functions $f$ and $g$,  ${\cal L}_{\nu_t} (f.g)= f {\cal L}_{\nu_t} g+g{\cal L}_{\nu_t} f+2\nabla f. \nabla g$. With the help of Cauchy-Schwarz inequality, this leads to 
\begin{align}\label{eq:pluggt3}
|{\rm Cov}_{\Phi(\nu_{t})}(f,g_t)|&\le \|\nabla \varphi_t\|_{L^2(\Phi(\nu_{t}))}\|\nabla g_t\|_{L^2(\Phi(\nu_{t}))}\\
&\le  \|\nabla \varphi_t\|_{\infty}\|\nabla g_t\|_{\infty}.\nonumber
\end{align}

\noindent\textbf{Control of $\|\nabla \varphi_t\|_{\infty}$.} Let $h:\mathbb{R}^d\rightarrow\mathbb{R}$ be a ${\cal C}^1$ function with bounded derivative. Let $\varphi$ be a unique solution of the Poisson equation $h-\langle \Phi(\mu), h\rangle={\cal L}_\mu \varphi$. It is well-known that $\varphi(x)=\int_0^{+\infty} P_t^{\mu} h(x)-\langle\Phi(\mu),h\rangle dt$ where $(P_t^{\mu})_{t\ge0}$ denotes the semi-group associated with ${\cal L}_\mu$ so that under adequate derivation conditions (which will be satisfied in our setting), 
$$\nabla \varphi(x)=\int_0^{+\infty} \nabla P_t^{\mu} h(x) dt.$$
Indeed, it is sufficient to verify the conditions of application of the Lebesgue dominated convergence theorem for the derivation under the integral sign, which is done below.

We have 
$$ \nabla P_t^{\mu} h (x)=\mathbb{E}[\nabla h(X_t^{\mu,x}) \partial_x X_t^{\mu,x}]$$
where $(X_t^{\mu,x})$ denotes the solution starting from $x$ of the SDE associated with ${\cal L}_\mu$ and $(\partial_x X_t^\mu)$ its first variation process solution to 
$$ dY_t=-\nabla\mathcal{D}_{m}H(\mu,X^{\mu,x}_{t}) Y_t dt$$
with $Y_0=I_d$. Under the assumption 
\begin{equation}
    \underline{\lambda}:=\inf_{\mu,x}\lambda_{\nabla\mathcal{D}_{m}H(\mu,x)+\lambda_\mu\nabla\mathcal{D}_{m}E(\mu,x)}>0,
\end{equation}
one easily deduces from a Gronwall argument that for any $z\in\mathbb{R}^d$,
$$ |Y_t z|^2\le e^{-2\underline{\lambda}t}|z|^2$$
which implies that 
$$|\nabla P_t^{\mu} h(x)|\le [h]_1\|\partial_x X_t^{\mu,x}\|_{\mathbf{op}}\le [h]_1 e^{-\underline{\lambda}t}$$
where $[h]_1$ denotes the Lipschitz constant of $h$.
Then,
$$\|\nabla \varphi\|_\infty\le [h]_1\underline{\lambda}^{-1}.$$
In particular, we have
\begin{equation}
    \|\nabla \varphi_t\|_{\infty}\leq[f]_1\underline{\lambda}^{-1}=\frac{1}{\underline{\lambda}}.
\end{equation}

\noindent\textbf{Control of $\|\nabla g_t\|_{\infty}$.} Let us finally focus on $\nabla g_t$. First, one checks that
\begin{equation}
    \nabla g_t(x):=-\frac{1}{\sigma^2}\bigg(\int_{\mathbb{R}^{d}}\nabla_x\frac{\delta^2 H}{\delta m^2}(\nu_{t},x,y)(\nu_{1}-\nu_{0})(dy)+\lambda_{\nu_t}\int_{\mathbb{R}^{d}}\nabla_x\frac{\delta^2 E}{\delta m^2}(\nu_{t},x,y)(\nu_{1}-\nu_{0})(dy)+\nabla\frac{\delta E}{\delta m}(\nu_t,x)\frac{d}{dt}\lambda_{\nu_t}\bigg).
\end{equation}
Now, denoting by $(X_{\nu_0},X_{\nu_1})$ an optimal coupling of $\nu_0$ and $\nu_1$ for the $1$-Wasserstein distance, one obtains
\begin{align}
\bigg|\nabla_x \left(\int_{\mathbb{R}^{d}}\frac{\delta^2 H}{\delta m^2}(\nu_{t},x,y)(\nu_{1}-\nu_{0})(dy)\right)\bigg|&\le  |\mathbb{E}[\nabla_{x}\frac{\delta^2 H}{\delta m^2}(\nu_{t},x,X_{\nu_{1}}) -\nabla_{x}\frac{\delta^2 H}{\delta m^2}(\nu_{t},x,X_{\nu_{0}})]|\\
&\le M^H_2\mathcal{W}_{1}(\nu_{0},\nu_{1});\nonumber\\
\bigg|\nabla_x \left(\int_{\mathbb{R}^{d}}\frac{\delta^2 E}{\delta m^2}(\nu_{t},x,y)(\nu_{1}-\nu_{0})(dy)\right)\bigg|&\le  |\mathbb{E}[\nabla_{x}\frac{\delta^2 E}{\delta m^2}(\nu_{t},x,X_{\nu_{1}}) -\nabla_{x}\frac{\delta^2 E}{\delta m^2}(\nu_{t},x,X_{\nu_{0}})]|\\
&\le M^E_2\mathcal{W}_{1}(\nu_{0},\nu_{1}).\nonumber
\end{align}
Hence, 
$$ \|\nabla g_t\|_{\infty}\le\frac{M^H_2+|\lambda_{\nu_t}|M^E_2}{\sigma^2}\mathcal{W}_{1}(\nu_{0},\nu_{1})+\frac{M^E_1}{\sigma^2}\bigg|\frac{d}{dt}\lambda_{\nu_t}\bigg|$$
and by \eqref{eq:pluggt3}, we get for any smooth $1$-Lipschitz function
$$|\langle\Phi(\nu_{1}),f\rangle-\langle\Phi(\nu_{0}),f\rangle|\leq\int_{0}^{1}\bigg|\frac{d}{dt}\langle\Phi(\nu_{t}),f\rangle\bigg| dt\le\frac{M^H_2+M^E_2\int_{0}^1|\lambda_{\nu_t}|dt}{\sigma^2\underline{\lambda}}\mathcal{W}_{1}(\nu_{0},\nu_{1})+\frac{M^E_1}{\sigma^2\underline{\lambda}}\int_0^1\bigg|\frac{d}{dt}\lambda_{\nu_t}\bigg|dt.$$
By \cref{FlatDerivativeLambda}, we have
\begin{equation}
    |\langle\Phi(\nu_{1}),f\rangle-\langle\Phi(\nu_{0}),f\rangle|\leq\frac{M^H_2+C_\lambda M^E_2+\|\nabla\frac{\delta\lambda}{\delta m}\|_{\infty}M^E_1}{\sigma^2\underline{\lambda}}\mathcal{W}_1(\nu_0,\nu_1).
\end{equation}
We conclude by a density argument and the Kantorovitch-Rubinstein duality relation. Concerning the existence of at most one fixed point, we show it by the absurd: assuming that we have more than one, we would contradict the strict contraction since we would have
\begin{equation}
    \frac{M^H_2+C_\lambda M^E_2+\|\nabla\frac{\delta\lambda}{\delta m}\|_{\infty}M^E_1}{\sigma^2\underline{\lambda}}<1.
\end{equation}
Existence follows from Picard’s fixed point theorem since $\Phi$ is a strict contraction of $(\mathcal{C}_M,\mathcal{W}_1)$ which is a complete metric space.
\end{proof}

\begin{proof}[Proof of \cref{Theorem9}]
Using Cauchy-Schwarz, Young and log-Sobolev inequalities, we have
\begin{align}
    \frac{d}{dt}\bigg(\mathcal{E}(\mu_t)-\mathcal{E}(\mu_\infty)\bigg)&=\int\frac{\delta\mathcal{E}}{\delta m}(\mu_t,\cdot)d\partial_t\mu_t\\
    &=-\int\nabla\frac{\delta\mathcal{E}}{\delta m}(\mu_t,\cdot)\cdot\nabla\bigg(\frac{\delta\mathcal{E}}{\delta m}(\mu_t,\cdot)+\lambda_t\frac{\delta E}{\delta m}(\mu_t,\cdot)\bigg)d\mu_t\nonumber\\
    &=-\int\bigg\|\nabla\frac{\delta\mathcal{E}}{\delta m}(\mu_t,\cdot)\bigg\|^2_2d\mu_t-\lambda_t\int\nabla\frac{\delta\mathcal{E}}{\delta m}(\mu_t,\cdot)\cdot\nabla\frac{\delta E}{\delta m}(\mu_t,\cdot)d\mu_t\nonumber\\
    &\leq-\int\bigg\|\nabla\frac{\delta\mathcal{E}}{\delta m}(\mu_t,\cdot)\bigg\|^2_2d\mu_t+|\lambda_t|M^E_1\sqrt{\int\bigg\|\nabla\frac{\delta\mathcal{E}}{\delta m}(\mu_t,\cdot)\bigg\|^2_2d\mu_t}\nonumber\\
    &\leq-\frac{1}{2}\int\bigg\|\nabla\frac{\delta\mathcal{E}}{\delta m}(\mu_t,\cdot)\bigg\|^2_2d\mu_t+\frac{(|\lambda_t|M^E_1)^2}{2}\nonumber\\
    &\leq-\rho_{LS}\bigg(\mathcal{E}(\mu_t)-\mathcal{E}(\mu_\infty)\bigg)+\frac{(|\lambda_t|M^E_1)^2}{2}.\nonumber
\end{align}
By Gronwall's lemma, we have
\begin{equation}
    \mathcal{E}(\mu_t)-\mathcal{E}(\mu_\infty)\leq\bigg(\mathcal{E}(\mu_0)-\mathcal{E}(\mu_\infty)\bigg)e^{-\rho_{LS}t}+\frac{(M^E_1)^2}{2}\int_0^t|\lambda_s|^2e^{-\rho_{LS}(t-s)}ds.
\end{equation}
\end{proof}

\begin{proof}[Proof of \cref{Theorem10}]
    By \cref{PropoEntropyProd}, we have
    \begin{equation}
        \frac{d}{dt}\bigg(\mathcal{E}(\mu_t)-\mathcal{E}(\mu_\infty)\bigg)=-\mathcal{I}(\mu_t)\leq-2\rho_{LS}\bigg(\mathcal{E}(\mu_t)-\mathcal{E}(\mu_\infty)\bigg).
    \end{equation}
    We conclude with Gronwall's lemma.
\end{proof}

\begin{proof}[Proof of \cref{Theorem11}]
    By \cref{propoFreeMicroscopicEnergy}, we have
    \begin{equation}
        \frac{d}{dt}\mathbf{H}[\mu^n_t|\mu^n_\infty]=-\mathbf{I}[\mu^n_t|\mu^n_\infty]\leq-2\rho_{LS}(\mu^n_\infty)\mathbf{H}[\mu^n_t|\mu^n_\infty].
    \end{equation}
    We conclude with Gronwall's lemma and Talagrand $T_2$-inequality.
\end{proof}

\addcontentsline{toc}{section}{References}
\bibliographystyle{unsrt}  
\bibliography{references}

\end{document}